\documentclass[10pt,a4paper]{article}
\usepackage[latin1]{inputenc}
\usepackage{amsmath}

\usepackage{apalike}

\usepackage{graphicx,psfrag,epsf}
\usepackage{enumerate}
\usepackage{rotating}

\usepackage{amsfonts}
\usepackage{amssymb}
\usepackage{graphicx}
\usepackage{graphicx}
\usepackage{amsmath}
\usepackage{amssymb}
\makeatletter

\usepackage{csquotes}

\usepackage{amsmath}
\usepackage{graphicx,psfrag,epsf}
\usepackage{enumerate}

\usepackage{url}
\usepackage{siunitx}

\usepackage{graphicx}
\usepackage{color}
\usepackage{lipsum}
\usepackage{url}
\usepackage{listings}
\usepackage{hyphenat}

\usepackage{threeparttable}
\usepackage{amssymb}
\usepackage{graphicx}
\usepackage{color}
\usepackage{lipsum}
\usepackage{url}
\usepackage{listings}
\usepackage{hyphenat}

\usepackage{makeidx}  
\usepackage{amssymb}
\usepackage{amsmath}
\usepackage{pifont}
\usepackage{lipsum}
\usepackage{multirow}
\usepackage{booktabs}
\usepackage{color}
\usepackage{soul}
\usepackage{url}
\usepackage{pifont}
\usepackage{lipsum}
\usepackage{multirow}
\usepackage{booktabs}
\usepackage{color}
\usepackage{soul}
\usepackage{url}

\usepackage{amsmath}
\usepackage{amsfonts}
\usepackage{amssymb}
\usepackage{graphicx}
\usepackage[T1]{fontenc}
\usepackage{selinput}

\usepackage[xindy]{glossaries}

\usepackage{amsmath}
\usepackage{amsfonts}
\usepackage{amssymb}
\usepackage{graphicx}

\usepackage{amsmath}

\usepackage{graphicx,psfrag,epsf}
\usepackage{enumerate}

\usepackage{amsfonts}
\usepackage{amssymb}
\usepackage{graphicx}

\usepackage{amsmath}
\usepackage{amssymb}
\makeatletter

\usepackage{amsthm}
\usepackage{csquotes}

\usepackage{amsmath}
\usepackage{graphicx,psfrag,epsf}
\usepackage{enumerate}

\usepackage{url}
\usepackage{siunitx}

\usepackage{graphicx}
\usepackage{color}
\usepackage{lipsum}
\usepackage{url}
\usepackage{listings}
\usepackage{hyphenat}

\usepackage{threeparttable}
\usepackage{amssymb}
\usepackage{graphicx}
\usepackage{color}
\usepackage{lipsum}
\usepackage{url}
\usepackage{listings}
\usepackage{hyphenat}

\usepackage{makeidx}  
\usepackage{amssymb}
\usepackage{amsmath}
\usepackage{pifont}
\usepackage{lipsum}
\usepackage{multirow}
\usepackage{booktabs}
\usepackage{color}
\usepackage{soul}
\usepackage{url}
\usepackage{pifont}
\usepackage{lipsum}
\usepackage{multirow}
\usepackage{booktabs}
\usepackage{color}
\usepackage{soul}
\usepackage{url}

\usepackage{amsmath}
\usepackage{amsfonts}
\usepackage{amssymb}
\usepackage{graphicx}
\usepackage[T1]{fontenc}
\usepackage{selinput}

\usepackage[xindy]{glossaries}

\usepackage{hyperref}
\usepackage{xcolor}
\hypersetup{
	colorlinks   = true, 
	urlcolor     = red, 
	linkcolor    = blue, 
	citecolor   = blue 
}

\usepackage{tikz}
\newcommand{\Mypm}{\mathbin{\tikz [x=1.4ex,y=1.4ex,line width=.1ex] \draw (0.0,0) -- (1.0,0) (0.5,0.08) -- (0.5,0.92) (0.0,0.5) -- (1.0,0.5);}}%

\makeglossaries

\newcommand{\distas}[1]{\mathbin{\overset{#1}{\kern\z@\sim}}}%

\newtheorem{theorem}{Theorem}

\title{Energy distance and kernel mean embedding for two sample survival test}

\usepackage{url}
\author{Marcos Matabuena\\ 
	marcos.matabuena@usc.es \\
	Centro de Investigaci\'{o}n en Tecnolox\'{i}as da Informaci\'{o}n (CiTIUS),\\ Universidade de Santiago de Compostela, Santiago de Compostela. Spain }

\begin{document}
	\maketitle
	\section{Abstract}
	In this article a new family of tests is proposed for the comparison problem of the equality of distribution of two-sample under right censoring scheme. The 
	tests are based on energy distance and kernels mean embedding, are calibrated by permutations and are consistent against all alternatives. The good performance of the new tests  in real situations with finite samples is established with a simulation study.
	
	{\it:} 
	\vfill

	\section{Introduction}

	One of the main objectives of the survival analysis is to compare the distribution of  the lifetime of two-sample coming of two different groups. The most popular example of this situation is the case of clinical trials when evaluating the efficacy of two treatments \cite{biblsingh2011survival}. Under a context of right censored data, the test most used within the scientific community to contrast the equality between two distribution curves is the logrank-test \cite{schoenfeld1981asymptotic,yang2010improved,su2018time} proposed at $1959$ by Mantel and Haenzel \cite{mantel1959statistical}. This test is known to be the most powerful test when the hazard functions are proportional to each other \cite{schoenfeld1981asymptotic,su2018time,xu2017designing}. However, when this hypothesis is violated the test has a significant loss of power \cite{fleming1980modified,lachin1986evaluation,lakatos1988sample,schoenfeld1981asymptotic}.


	Currently a hot topic in the medical field is for the hypothesis test to use \cite{su2018time}, due to the lack of statistical power of the log-rank test found in many real case studies \cite {su2018time}. This is the case of the new oncological treatments where, for example, with new immunotherapy therapies, they  have a delayed effect  \cite{melero2014therapeutic,xu2017designing,xu2018designing,su2018time}. Also in the multimodal treatments \cite{moehler2007multimodal} where it is expected that the density function in many occasions present several mode, or in cases where healing occurs \cite{lopez2017nonparametric}. In any of these situations, the hypothesis of proportional risks is strongly unfulfilled.

	From a mathematical statistics point of view it is well-known that any test with finite samples has a poor behavior except in a finite number of directions. This means that in real scenarios we have no guarantee that one test will always be better than another. Precisely Janssen \cite{janssen2000global} proved that you can not expect to build a test with a high power, except in a space of finite dimension. However, this does not mean that you can not build tests with an acceptable power for a large number of alternatives and in situations of interest, the objective sought by the statistical community in recent decades.

	In the literature there are two types of different  tests: the directionals and the omnibus. The former seek maximum power in specific directions, while the latter are consistent against all alternatives. The most popular family of directional tests with  right censoring is that of the logrank-test \cite{fleming1987supremum}, to which the statistic of the logrank test is assigned a weight function that determines the optimality in certain directions \cite{gehan1965generalized,tarone1977distribution,peto1972asymptotically,fleming1981class}. On other occasions, within these tests the results of the individual tests are even combined to construct a global test \cite{bathke2009combined}, or the function of weights \cite{yang2010improved} is estimated, but this needs a significant amount of data. The Kolmogorov-Smirnov \cite{fleming1980modified} test and the Cramer-von Mises with censorship on the right \cite{schumacher1984two} are two examples of omnibus tests.

	The energy distance \cite{szekely2003statistics,szekely2013energy} is a statistical distance that measures how many different two probability distributions are. It is based on the calculation of Euclidean distances between pairs of variables and the notion of potential energy, and it has been used among other problems to compare the equal distribution in problems with several samples \cite{szekely2004testing}, goodness of fit \cite{szekely2005new}, and cluster analysis \cite{szekely2005hierarchical}. The main characteristic of this statistic is that it requires minimum hypotheses for its use, only conditions on the moments of the random variables involved. Its multivariate extension is immediate, and the test for the comparison of equality in distribution in problems with several samples presents a high statistical power with known distributions, even in high-dimensional contexts  \cite{szekely2004testing}, being consistent for all alternatives. The generalization of the test with other types of metrics than the Euclidean ones like the negative type \cite{lyons2013distance,rachev2013methods} is equivalent to the methods kernel \cite{sejdinovic2013equivalence,shen2018exact} proposed in \cite{gretton2012kernel} and based on the kernel mean embedding \cite{muandet2017kernel}.

	The main objective of this paper is to extend these tests to a context of right censored data in the univariate case. The structure of the paper is as follows: first we review the main literature of the methods of comparison of equality in the distribution  of two-sample with right censoring, then we explain the relationship between energy distance and kernels mean embedding. The statistics are then derived and their theoretical properties of the test are established as the consistency against all alternatives. Finally a simulation study is carried out to compare the behavior of the proposed new methods against the classical tests of the literature. To do this, we will compare the power and error type I using known distributions, in addition to the cases discussed above, with delay, recovery or multimodality, where the log-rank test have less than ideal performance.

	\section{Previous research}
	
	Henceforth, let us consider the traditional framework in the problems of two-sample survival comparison  given by the lifetimes $T_{ji} \distas{} P_j$ $(j=0,1; i=1,\dots,n_j)$ and censoring times $C_{ji} \distas{}$ $Q_j$ $(j=0,1; i=1,\dots,n_j)$ with distributions   $P_j$ y $Q_j$  $(j=0,1)$ defined in an subset of $\mathbb{R^{+}}$. As usual, the random variables  $T_{01},\dots,T_{0n_0},\dots,T_{11},\dots,T_{1n_1},$ $ C_{01},\dots,C_{0n_0},\dots,C_{11},\dots,C_{1n_1}$ are assumed to be independent of each other. In practice only the random variables are observed  $X_{ji}= min(T_{ji},C_{j,i})$  and $\delta_{ji}= 1\{X_{ji} = T_{ji}\}$	$(j=0,1; i=1,\dots,n_j)$. We will always assume $ E(T_{ji}^{2})<\infty $ and $ E(C_{ji}^ {2}) <\infty$, and that the variables $X_{ji}, T_{ji},C_{j,i}$ $(j=0,1; i=1,\dots,n_j)$ are continuous   for simplification.

	The problem of two-sample that we will study is the following:
	
	\begin{equation}
	H_0: P_0(t)= P_1(t) 
	\hspace{0.2cm} \forall  t> 0 \text{ versus }    H_a: \exists t>0 \text{ such that }  P_0(t)\neq  P_1(t).
	\end{equation}
	
	At the maximum times observed for each group we will call them $\tau_0$ and $\tau_1$ respectively, and at the minimum of both, $\tau=\min(\tau_0,\tau_1)$.

	Next, we will describe the previous main literature on directional and omnibus tests.

	\subsection{Directional tests: The log-rank test family}

	In this subsection we will describe the logrank test and its different variants.

	The times of failure will be denoted as $\tau_1<\tau_2\cdots<\tau_k$. We define:  
	
	$Y_i(\tau_j)$=  $\#$ people in the group  $i$ who are at risk in  $\tau_j$ $(i=0,1;j=1,2\dots,k)$.
	
	$Y(\tau_j)= Y_0(\tau_j)+ Y_1(\tau_j)$=  $\#$ people at risk in  $\tau_j$ (in both groups).
	
	$d_{ij}=$ $\#$ people who fail in the group $i$ in $\tau_j$ $(i=0,1;j=1,2\dots,k)$.
	
	$d_j= d_{0j}+d_{1j}=$  $\#$ people who fail in $\tau_j$.
	
	The statistic has the following structure:

	\begin{equation}
	\hat{Z}^{2}= \frac{[\sum_{j=1}^{k}\omega_{j}(d_{1j}-E(d_{1j}))]^2}{\sum_{j=1}^{k} Var(\omega_{j}(d_{1j}-E(d_{1j})))}
	\end{equation}

	where $\omega_j$ $(j=1,2,\dots,k)$ is a weighting function that determine the properties of the test, and that depends on the number of people at risk in time $j$, $Y(\tau_j)$, of the survival function estimated in time  $j$ $\hat{S}(\tau_j)$,  or in the last instant $\hat{S}(\tau_{j-1})$.

	Under the null hypothesis $H_0: P_0(\cdot)= P_1(\cdot)$, $d_{1j}\approx H(Y(\tau_j),Y_1(\tau_j),d_j)$, where $H$ denotes the hypergeometric distribution and therefore, it is fulfilled, $E(d_{1j})= \frac{d_{1j}}{Y_j}Y_{1j}$ y $Var(d_{1j}-E(d_{1j}))= \frac{d_{j}(Y_{1j}/Y_j)(1-Y_{1j}/Y_j)(Y_{j}-d_j)}{Y_{j}-1}$·


	\begin{table}[h!]
		
		\begin{center}

			\caption{Possible events on time $\tau_k$}

			\begin{tabular}{||c c c c||} 
				\hline
				& Group $0$ &  Group $1$ & Total \\ [0.5ex] 
				\hline\hline
				Number of live subjects  & $Y_0(\tau_j)$ & $Y_1(\tau_j)$ & $Y(\tau_j)$\\ 
				\hline
				Number of subjects that die  & $d_{0j}$ & $d_{1j}$ & $d_j$ \\
				\hline
			\end{tabular}
		\end{center}
	\end{table}


	The main characteristics of the log-rank test and its variants will be described below:


	\begin{itemize}
		\item \textbf{Log-rank} \cite{mantel1959statistical}.
		
		The logrank test is optimal when the hazard function of the two groups are proportional. It results from taking $\omega_j=1$ $(j=1,2,\dots,k)$. Under the null hypothesis is fulfilled:

		\begin{equation*}
		\hat{Z}^{2}= \frac{[\sum_{j=1}^{k}(d_{1j}-E(d_{1j}))]^2}{\sum_{j=1}^{k} Var(d_{1j}-E(d_{1j}))}\overset{d}{\to} \chi^{2}_{1}.
		\end{equation*}
		
		
		\item \textbf{Gehan Generalized Wilcoxon Test} \cite{gehan1965generalized}
		
		It is a test of free distribution that is an extension of the Wilconxon test in a context of right-censored. It provides much more weight to the early survival times. For this, it is taken as a function of weights $\omega_j= Y(\tau_j)$ $(j=1,2,\dots,k)$.

		%

		\item \textbf{Tarone-Ware} \cite{tarone1977distribution}
		
		It is a modification of the Gehan test, whose weight function is $\omega_j= \sqrt{(Y(\tau_j)}$ $(j=1,2,\dots,k)$, which assigns lower weights than in the Gehan test.

		%
		
		\item \textbf{Peto-Peto} \cite{peto1972asymptotically}
		
		The Peto test is used when the hazard function is not proportional, and the Kaplan-Meier estimator is used in the weight function $\omega_j= \hat{S}(t_j)$ $(j=1,2,\dots,k)$. The initial times receive more weighting than the more distant observations.   
		
		\item \textbf{Fleming $\&$ Harrington family $G^{\rho,\gamma}$} \cite{fleming1981class}
		
		In the test family Fleming $\&$ Harrington $G^{\rho,\gamma}$   the function of weights $\omega_j= \hat{S}(t_{j-1})^{\rho}(1-\hat{S}(t_{k-1}))^{\gamma}$ $(j=1,2,\dots,k)$ depends on two parameters $\rho \geq 0$ y $\gamma \geq 0$ that give the test much flexibility. The choice as plug-in of the Kaplan-Meier estimator increases the power of the test \cite{buyske2000class}.

	\end{itemize}

	%

	%
	%
	%
	%
	%
	%

	\subsection{The omnibus tests}

	The  Kolmogorov Smirnov \cite{fleming1980modified,schumacher1984two} and Cramér-von Mises tests \cite{schumacher1984two}  under right-censored data are the
	most popular omnibus test.  There are several versions of these two tests but some have certain limitations. For example, the direct extension of the   Cramér-von Mises test to the censored case, the limit distribution \cite{koziol1978two} of the Cramér-von Mises in general can not be calculated. In this subsection we will explain two versions of both tests proposed in \cite{schumacher1984two} and based on the comparison of cumulative empirical hazard function.

	Suppose   $X_{ji}= min(T_{ji},C_{j,i})$  and $\delta_{ji}= 1\{X_{ji}= T_{ji}\}$	$(j=0,1; i=1,\dots,n_j)$ under the conditions of independence assumed in the section $2$ on the variables  $T_{ji},C_{j,i}$ $(j=0,1; i=1,\dots,n_j)$.

	To the ordered sample we will call them $X_{0(1:n_0)}$, $X_{0(2:n_0)},\dots, X_{0(n_0:n_0)}$,
	$X_{1(1:n_1)}$,$X_{1(2:n_1)},\dots X_{1(n_1:n_1)}$ and we will also refer to the corresponding censorship with respect to induced ordering for observed times $\delta_{0(1:n_0)}$, $\delta_{0(2:n_0)},\dots \delta_{0(n_0:n_0)}$,
	$\delta_{1(1:n_1)}$, $\delta_{1(2:n_1)},\dots \delta_{1(n_1:n_1)}$.

	Denoting by $S_0(t)$, $S_1(t)$  to the survival functions of the groups $0$ and $1$ respectively at the time instant $t$  and   $\Lambda_0(t)$, $\Lambda_1(t)$ to its cumulative hazard function, and considering the function

	%

	\begin{equation*}
	\epsilon(t)= \Lambda_1(t)-\Lambda_0(t)= -\log(S_1(t))+\log(S_0(t))).
	\end{equation*}

	The comparison problem $(1)$ can be expressed as:

	\begin{equation}
	H_0: \epsilon(t)= 0 
	\hspace{0.2cm} \forall  t> 0 \text{ versus }    H_a: \exists t>0 \text{ such that }  \epsilon(t)\neq 0.
	\end{equation}

	The function $\epsilon(t)$ can be estimated by:

	%

	\begin{equation*}
	\hat{\epsilon}(t)=  \hat{\Lambda}_1(t)-\hat{\Lambda}_0(t)
	\end{equation*}


	where, $\hat{\Lambda}_j(t)= \sum_{i:\hspace{0.1cm} \tau_i\leq t}\frac{d_{ji}}{Y_j(\tau_i)}$ $(j=0,1)$,  
	denotes the estimator of Nelson-Aalen
	
	\cite{nelson1972theory} of each group.      
	
	We define:

	$Y_j(t)= \sum_{i=1}^{n_j} 1\{X_{j(i:n_j)}\geq t\} \hspace{0.2cm} (j=1,2), $
	
	$\hat{A_j}(t)= n_j\sum_{i: \hspace{0.1cm} X_{j(i:n_j)\leq t}}\frac{\delta_{j:(i:n_j)}}{Y_j(X_{j(i:n_j)})[Y_{j(i:n_j)}+1]} \hspace{0.2cm} (j=1,2),$
	
	$\hat{A}(t)= \frac{n_0+n_1}{n_0}\hat{A}_0(t)+\frac{n_0+n_1}{n_1}\hat{A}_1(t),$

	$\hat{H}(t)=\hat{A}(t)/(1+\hat{A}(t)),$

	$\hat{\psi}_{\epsilon}(t)= 1
	/(\hat{A}(\tau))^{\frac{1}{2}},$

	$\hat{\psi}^{0}_{\epsilon}(t)= 1
	/(1+\hat{A}(t))$.
	

	From the previous expressions we can write the following two statistics of the Kolmogorov test

	\begin{equation}
	Q_{KS\epsilon}= (n_0+n_1)\sup_{0\leq t\leq\tau} |\hat{\epsilon}(t)\hat{\psi_\epsilon}(t)| \hspace{0.2cm} 	Q^{0}_{KS\epsilon}= (n_0+n_1)\sup_{0\leq t\leq\tau} |\hat{\epsilon}(t)\hat{\psi}^{0}_\epsilon(t)|,
	\end{equation}

	and also for the Cramér?von Mises test:

	\begin{equation}
	Q_{CM\epsilon}= ((n_0+n_1)/\hat{A}(\tau))\int_{0}^{\tau} (\hat{\epsilon}(t)\hat{\psi_\epsilon}(t))^{2}d\hat{A}(t) \hspace{0.2cm} 	Q^{0}_{CM\epsilon}= (n_0+n_1)\int_{0}^{\tau} (\hat{\epsilon}(t)\hat{\psi^{0}_\epsilon}(t))^{2}d\hat{H}(t).
	\end{equation}

	All statistics are consistent against all alternatives, and convergence almost surely  to their analogous populations. The limit distribution of Kolmogorov Smirnov tests    is in $Q_{KS\epsilon}$  the next 	Gaussian process $W(A(t)/A(\tau))$ where  $W(x)$  denotes a Brownian standard movement and $Q^{0}_{KS\epsilon}$ converge to $W^{0}(H(t))$ where $W^{0}(x)$ is a Brownian bridge. While 
	$Q_{CM\epsilon}$ converge to $A(\tau)^{2}\int_{0}^{\tau}[W(A(t))^{2}]dA(t))$ and $Q^{0}_{CM\epsilon}$ to $A(\tau)^{2}\int_{0}^{\tau}[W^{0}(H(t))^{2}]dH(t))$.
	
	For more details consult the following reference \cite{schumacher1984two}.

	
	%
	%
	%
	%
	%
	%
	%
	%
	%
	%
	%
	%
	%
	%
	%

	\section{The energy distance and the kernels mean embedding}

	In this section we will introduce the energy distance, the RKHS (reproducing kernel Hilbert space) and its relation with the kernels mean embeddings. The explanation will be first at the population level and then at the sample level.

	Given the random variables in $\mathbb{R}^{d}$ $X$,$X^{\prime}\distas{iid}P$ and $Y$,$Y^{\prime}\distas{iid}Q$, with finite moments of order one $E(||X||_d)\leq \infty$, $E(||X^\prime||_d)\leq \infty$, $E(||Y||_d)\leq \infty$,$E(||Y^\prime||_d)\leq \infty$, and where,  $P$ y $Q$ denotes its distribution functions. The energy distance \cite{szekely2003statistics,szekely2013energy} between the distributions $P$ and $Q$ is defined by:

	\begin{equation}
	\epsilon(P,Q)= 2E||X-Y||- E||X-X^{'}||-E||Y-Y^{'}||,
	\end{equation}
	
	where $||\cdot||$ denotes the Euclidean norm.

	It can be proved that $\epsilon(P,Q)$ it is invariant before rotations, in addition, it is non-negative $\epsilon(P,Q)\geq 0$, giving equality to zero, if and only, $P=Q$.	
	
	The previous definition of energy distance can be extended for a family of indices $\alpha\in(0,2]$ \cite{szekely2013energy} (assuming in each case the existence of the moment of order $\alpha$). In this case, the $\alpha$ energy distance is:

	\begin{equation}
	\epsilon_\alpha(P,Q)= 2E||X-Y||^{\alpha}- E||X-X^{'}||^{\alpha}-E||Y-Y^{'}||^{\alpha}.
	\end{equation}

	verifying, for all $\alpha\in(0,2)$ $\epsilon_\alpha(P,Q)\geq 0$, and giving equality to zero, if and only, $P=Q$. In the particular case with $\alpha=2$, $\epsilon_2(P,Q)= 2||E(X)-E(Y)||^{2}$, 	and therefore, non-negativity is verified trivially, although in this situation,	 	$\epsilon_2(P,Q)=0$, implies equality in means and not in distribution between $P$ and $Q$.
	
	The notion of energy distance can be generalized to even more general spaces. Let $X,Y\in V$ where $V$ is an arbitrary space with a scalar product induced by a semi-metric of negative type \cite{rachev2013methods,lyons2013distance} $\rho: V\times V\to \mathbb{R}$, what is required to satisfy:

	\begin{equation}
	\sum_{i,j=1}^{n} c_i c_j \rho(X_i,X_j)\leq 0
	\end{equation}   	
	
	where $\forall X_i,X_j\in V$, and each $c_i\in \mathbb{R}$ such that $\sum_{i=1}^{n} c_i=0$. In this case, the pair $(V,\rho)$ it is said to be a negative type space \cite{lyons2013distance,rachev2013methods}. Replacing $\mathbb{R}^{d}$ by $V$ and $||X-Y||$ by $\rho(X,Y)$, in expression $(6)$, we obtain the generalized energy distance for the negative type space $(V,\rho)$:
	
	\begin{equation}
	\epsilon_\alpha(P,Q)= 2E\rho(X,Y)- E\rho(X,X^{\prime})-E\rho(Y,Y^{\prime}).
	\end{equation}
	
	In any negative type space $(V,\rho)$ there is a hilbert space $H$ and an application $\phi: V\to H$ such that $\rho(X,Y)=||\phi(X)-\phi(Y)||_{H}^{2}$\cite{rachev2013methods,sejdinovic2013equivalence}. The previous relationship allows calculating the amounts of the distributions on $V$ in the associated Hilbert space $H$. In the case $\rho$ does  does not satistate the triangular inequality, the function $\rho^{1/2}$ the function verifies the distance axioms.

	There is an equivalence \cite{szekely2013energy,shen2018exact} between energy distance, commonly used in statistics \cite{szekely2013energy}, and the distance defined in the kernels mean embeddings \cite{gretton2012kernel}, the approach used mostly in the field of machine learning \cite{gretton2012kernel}. Before explaining, we are going to introduce some basic concepts of the RKHS. For more information about the RKHS consult the following basic reference \cite{manton2015primer}.

	Let $H$ be the Hilbert space that contains the real variable functions defined above $V$. A function $K:V\times V\to \mathbb {R}$ is a reproducing kernel in $H$ if it satisfies the following two properties:

	\begin{enumerate}
		\item $K(\cdot,x)\in H$
		\item  $<K(\cdot,x),f>= f(x)$ $\forall x\in V$ and $f\in H$. 
	\end{enumerate}

	The two properties above imply that $K$ is a positive definite and symmetric function. The theorem of Moore-Aronszajn \cite{aronszajn1950theory,manton2015primer} establishes the converse equivalence, if $K: V\times V\to \mathbb {R}$ is a symmetric function and positive definite, there is a single  reproducing kernel Hilbert space $H_K $, which has as its reproducing kernel $K$. The application $\phi: x\to K(\cdot,x) \in H_K$ is the so-called canonical feature application. Given a $K$ kernel, this theorem provides a method of how to define an embedding of a probability measure $P$ in an RKHS space. To do this, just consider the application $P\to h_P \in H_K$ such that $\int f(x)dP(x) = <f,h_P> $ $\forall f \in H_K$, or equivalently, define $h_P = \int_{}^{}K(\cdot, x)dP(x) $.

	The notion of distance between two probabilities can be introduced using the inner product of $H_K$, which, is called measure of maximum discrepancy (MMD) \cite{gretton2012kernel} and is given by:

	\begin{equation}
	\gamma_K(P.Q)= ||h_P-h_Q||_{H_K}.
	\end{equation}

	The above expression \cite {gretton2012kernel} can also be written as :

	\begin{equation}
	\gamma_K^{2}(P,Q)= E(K(X,X^{'}))+E(K(Y,Y^{'}))-2E(K(X,Y))
	\end{equation}

	where $X$,$X^{\prime}\distas{iid}P$ and $Y$,$Y^{\prime}\distas{iid}Q$.

	The next important result  shows that negative-type semimetrics and positive defined kernels are strongly connected \cite{van2012harmonic}. Let $\rho:V\times V\to \mathbb{R}$ and $x_0\in V$ an arbitrarily fixed point. If it is defined:
	
	\begin{equation*}
	K(x,y)= \frac{1}{2}[\rho(x,x_0)+\rho(y,x_0)-\rho(x,y)].
	\end{equation*}

	Then, it can be shown that $K$ is a positive defined kernel if and only $\rho$ is a semimetric of negative type. In this way, we have a family of kernels, one for each election of $K(\cdot,x_0)$. Conversely, if $\rho$ is semimetric of negative type and $K$ is a kernel in this family, then it is verified:

	\begin{equation*}
	\rho(x,y)= K(x,x)+K(y,y)-2K(x,y)=||h_x-h_y||^{2}_{H_k}
	\end{equation*}

	Finally using the above equality, along with $(10)$ and $(11)$ can be established
	the relation between the distance in the kernels mean embedding and the distance of energy in a space of negative type $(V,p)$ \cite{sejdinovic2013equivalence}:
	
	\begin{equation}
	\epsilon(P,Q)=2[E(K(X,X')+EK(Y,Y')-2EK(X,Y)]=2 \gamma_K^{2}(P,Q).
	\end{equation}

	In a sample context, two samples are available $\{X_{0i}\}_{i = 1}^{n_0}\distas{iid} P_0 $, $\{X_{1i}\}_{i = 1}^{n_1}\distas{iid}P_1$ and the unknown quantities $\epsilon (P_0,P_1) $ and $\gamma_K^{2}(P_0,P_1)$ must be estimated. To do this, the empirical distribution is used as a plug-in and the statistical $U$ and $V$ is used as estimator. That is:

	\begin{equation}
	\hat{\epsilon}_{\alpha}(P_0,P_1)= \frac{2}{n_0n_1}\sum_{i=1}^{n_0}  \sum_{j=1}^{n_1} ||X_{0i}-X_{1j}|| - \frac{1}{n_0^2}\sum_{i=1}^{n_0}  \sum_{j=1}^{n_0} ||X_{0i}-X_{0j}|| - \frac{1}{n_1^2}\sum_{i=1}^{n_1}  \sum_{j=1}^{n_1} ||X_{1i}-X_{1j}||  \hspace{0.2cm} 
	\end{equation}
	
	\textbf{($V$ statistic $\alpha$ energy distance)},

	\begin{equation}
	\hat{\gamma}_K^{2}(P_0,P_1)=  \frac{2}{n_0n_1}\sum_{i=1}^{n_0}  \sum_{j=1}^{n_1} K(X_{0i},X_{1j})- \frac{1}{n_0^2}\sum_{i=1}^{n_0}  \sum_{j=1}^{n_0} K(X_{0i},X_{0j}) + \frac{1}{n_1^2}\sum_{i=1}^{n_1}  \sum_{j=1}^{n_1} K(X_{1i},X_{1j}).
	\end{equation}	
	
	\textbf{($V$ statistic kernel method)},

	\begin{equation}
	\hat{\epsilon}_{\alpha}(P_0,P_1)=  \frac{1}{n_0(n_0-1)}\sum_{i=1}^{n_0}  \sum_{i\neq j}^{n_0} ||X_{0i}-X_{0j}|| - \frac{1}{n_1(n_1-1)}\sum_{i=1}^{n_1}  \sum_{j\neq i}^{n_1} ||X_{1i}-X_{1j}||-\frac{2}{n_0n_1}\sum_{i=1}^{n_0}  \sum_{j=1}^{n_1} ||X_{0i}-X_{1j}|| 
	\end{equation}
	\textbf{($U$ statistic $\alpha$ energy distance)},
	\begin{equation}
	\hat{\gamma}_K^{2}(P_0,P_1)=  \frac{1}{n_0(n_{0i}-1)}\sum_{i=1}^{n_0}  \sum_{i\neq j}^{n_0} K(X_{0i},X_{0j}) + \frac{1}{n_1(n_1-1)}\sum_{i=1}^{n_1}  \sum_{j\neq i}^{n_1} K(X_{1i},X_{1j})-\frac{2}{n_0n_1}\sum_{i=1}^{n_0}  \sum_{j=1}^{n_1} K(X_{0i},X_{1j}).
	\end{equation}
	\textbf{($U$ statistic kernel method)}.
	
	where the kernel  $K:\mathbb{R}^{d}\times \mathbb{R}^{d}\to \mathbb{R}$ it has to be characteristic \cite{sriperumbudur2011universality,gretton2012kernel,muandet2017kernel}.

	In the  table $2$ we can see the most known kernels with the property of being characteristic.

	\begin{table}[h!]
		
		\begin{center}
			\caption{\textbf{Characteristics kernels}.
				$\Gamma(\cdot)$ denote Gamma function and $K_v$ is the modified Bessel
				function of the second kind of order
				$v$.}
			\label{tab:table1}
			\begin{tabular}{c|c} 
				\textbf{Kernel Function} & \textbf{$k(x,y)$} \\ 
				\hline
				Gaussian & $\exp(-\sigma ||x-y||^{2}), \sigma>0$ \\ 
				Laplacian &  $\exp(-\sigma |x-y|),   \sigma>0$ \\
				Rational quadratic & $(||x-y||+c)^{-\beta}$, $\beta,\alpha>0$ \\
				Mattern   & $\frac{2^{1-v}}{\Gamma(v)}(\frac{\sqrt{2v}||x-y||}{\sigma})K_v(\frac{\sqrt{2v}||x-y||}{\sigma})$ \\
				
			\end{tabular}
		\end{center}
	\end{table}

	In the statistical community we usually use the energy of data with a $V$ statistic, which is a biased estimator \cite{kowalski2008modern}, but which is always greater than or equal to zero \cite{szekely2013energy}. While in the community of machine learning it is obtained by the kernel method with $U$ statistics, unbiased estimator \cite{kowalski2008modern}, with a lower computational cost, but which can take negative values \cite{gretton2012kernel}.

	Assuming moments of at least $2$ order in the random variables $\{X_{0i}\}_{i=1}^{n_0} \distas{iid}P_0$, $\{X_{1i}\}_{i=1}^{n_1}\distas{iid}P_1$, the sample statistic converges almost surely to the population version:

	\begin{equation}
	\hat{\epsilon}_{\alpha}(P_0,P_1)\overset{n_0,n_1\to \infty} \to \epsilon(P_0,P_1)\hspace{0.3cm} \alpha \in(0,2) , 
	\end{equation}
	
	\begin{equation}
	\hat{\gamma}_K^{2}(P_0,P_1)\overset{n_0,n_1\to \infty} \to \gamma_K^{2}(P_0,P_1).
	\end{equation}

	The limit distribution of these statistics is derived as a consequence of the central theorems for $U$ and $V$ statistics in the degenerate case \cite{korolyuk2013theory}  and can be found in the original works \cite{gretton2012kernel,szekely2004testing}. However, in practice, to calibrate the tests the boostrap/permutations methods are used \cite{gretton2012kernel,szekely2004testing}.

	\section{The proposed tests}
	
	In this section, the tests based on energy distance and kernel mean embedding will be extended to a context of right censoring. In this case, unlike the previous section, the statistics will be deducted first and then the theoretical properties will be derived.

	\subsection{The statistics}

	As before, let us suppose  $X_{ji}= min(T_{ji},C_{j,i})$  and $\delta_{ji}= 1\{X_{ji}= T_{ji}\}$	$(j=0,1; i=1,\dots,n_j)$ under the conditions of independence and regularity assumed in the section $ 2 $ on the variables $T_{ji},C_{j,i}$ $(j=0,1; i=1,\dots,n_j)$.

	For each group we consider their orderly sample $X_{0(1:n_0)}$, $X_{0(2:n_0)},\dots ,X_{0(n_0:n_0)}$,
	$X_{1(1:n_1)}$, $X_{1(2:n_1)},\dots ,X_{1(n_1:n_1)}$ and also for the corresponding censored indicators $\delta_{0(1:n_0)}$, $\delta_{0(2:n_0)},\dots \delta_{0(n_0:n_0)}$,
	$\delta_{1(1:n_1)}$, $\delta_{1(2:n_1)},\dots ,\delta_{1(n_1:n_1)}$.

	In a context of right censoring (under independence), the maximum non
	parametric 
	likelihood estimator is the Kaplan-Meier \cite{kaplan1958nonparametric} estimator instead of the empirical distribution. This estimator is consistent \cite{wang1987note} and for all $t>0 $,  converges asymptotically a normal distribution \cite{cai1998asymptotic}. One of its main characteristics is its negative bias \cite{stute1994bias}, which if it is a mechanism of censored is high it can become considerable. In \cite{stute1994bias} in fact, an exact expression is provided for the bias of the Kaplan-Meier integral $\int \phi d \hat{F}_n$, where $\hat{F}_n$ denotes Kaplan-Meier estimator.

	If we replace as plug-in, the empirical distribution by the Kaplan-Meier estimator in $(13)$ and $(14)$, we obtain the $V$ statistic for right censored data:
	
	\begin{align}
	\hat{\epsilon}_{\alpha}(P_0,P_1)= 2\sum_{i=1}^{n_0}  \sum_{j=1}^{n_1} W^{0}_{i:n_0} W^{1}_{i:n_1} ||X_{0i}-X_{1j}||^{\alpha} - \sum_{i=1}^{n_0} \sum_{j=1}^{n_0}  W^{0}_{i:n_0} W^{0}_{j:n_0}  ||X_{0i}-X_{0j}||^{\alpha} \\ - \sum_{i=1}^{n_1}  \sum_{j=1}^{n_1} W^{1}_{i:n_1} W^{1}_{j:n_1} ||X_{1i}-X_{1j}||^{\alpha} \nonumber,
	\end{align}
	\textbf{($V$ statistic energy distance under right censored)},
	\begin{align}
	\hat{\gamma}_K^{2}(P_0,P_1)=  \sum_{i=1}^{n_0}  \sum_{j=1}^{n_0} W^{0}_{i:n_0} W^{0}_{j:n_0} K(X_{0i},X_{0j})+ \sum_{i=1}^{n_1}  \sum_{j=1 i}^{n_1} W^{1}_{i:n_1} W^{1}_{j:n_1} K(X_{1i},X_{1j})\\-2\sum_{i=1}^{n_0}  \sum_{j=1}^{n_1} W^{0}_{i:n_0} W^{1}_{i:n_1} K(X_{0i},X_{1j}) \nonumber.
	\end{align}
	\textbf{($V$ statistic kernel method under right censored)}.


	where
	
	\begin{equation}
	W^{0}_{{i:n_0}}= \frac{\delta_{0(i:n_0)}}{n_0-i+1}\prod_{j=1}^{i-1}[\frac{n_0-j}{n_0-j+1}]^{\delta_{0(i:n_0)}}
	\hspace{0.2cm}   (i=1,\dots,n_0)
	\end{equation}
	
	and

	\begin{equation}
	W^{1}_{{i:n_1}}= \frac{\delta_{1(i:n_1)}}{n_1-i+1}\prod_{j=1}^{i-1}[\frac{n_1-j}{n_1-j+1}]^{\delta_{1(i:n_1)}} \hspace{0.2cm}   (i=1,\dots,n_1),
	\end{equation}

	are the Kaplan-Meier integral weights \cite{stute1995statistical}.

	However, as the limit of each statistic has the following structure:

	\begin{align}
	\hat{\epsilon}_\alpha(P_0,P_1) \overset{n_0,n_1\to \infty}{\to} \epsilon_{c(\alpha)}(P_0,P_1)= 2\int_{0}^{\tau_0} \int_{0}^{\tau_1}	 ||x-y||^{\alpha} dP_0^{\prime}(x) dP_1^{\prime}(y)\\
	-\int_{0}^{\tau_0} \int_{0}^{\tau_0}	 ||x-y||^{\alpha} dP_0^{\prime}(x) dP_0^{\prime}(y)-\int_{0}^{\tau_1} \int_{0}^{\tau_1}	 ||x-y||^{\alpha} dP_1^{\prime}(x) dP_1^{\prime}(y) \nonumber,
	\end{align} 
	
	\begin{align}
	\hat{\gamma}_K(P_0,P_1) \overset{n_0,n_1\to \infty}{\to} \gamma_{c(K)}(P_0,P_1)= 
	\int_{0}^{\tau_0} \int_{0}^{\tau_0}	 K(x,y) dP_0^{\prime}(x) dP_0^{\prime}(y) \\ +\int_{0}^{\tau_1} \int_{0}^{\tau_1}	K(x,y) dP_1^{\prime}(x) dP_1^{\prime}(y)-2\int_{0}^{\tau_0} \int_{0}^{\tau_1}	 K(x,y) dP_0^{\prime}(x) dP_1^{\prime}(y) \nonumber,
	\end{align} 
	
	where usually, $\tau_0$, $\tau_1$ are less than the maximum support value of the random variables  $P_0$ and $P_1$ due to censorship. $P^{\prime}_0$ and $P^{\prime}_1$ take values in this domain with the same value that initial distribution $P_0$ and $P_1$, but in general they not are distribution functions in the previous domain of  integration. 
	
	As a consequence, there is no guarantee that the limit functions $\gamma_{c(K)}(P_0,P_1)$ and $\epsilon_{c(\alpha)}(P_0,P_1)$ are a function of distance between probability measures. Actually they are not, if $P^0$ is the distribution function of a random variables $N(100000,1)$,  $P^1$  of a $Uniform(0,1)$ and $\tau_0=\tau_1 =0.1$, then the value of $\epsilon_{c(\alpha)}(P_0, P_1)$ is negative. It is easy to verify that if $P_0 = P_1 $ and $\tau_0 = \tau_1 $, then the limit is zero, but also we can build an example of two different probability measures with zero distance, so this statistics will not be consistent against all alternatives.
	
	To solve this problem, we have to get $P^{\prime} _0 (x)$, $P^{\prime}_1 (x)$ to be distribution functions in the previous integration domain, that is achieved by  the previous functions, that is, $P^{\prime\prime}_0 (x)= P^{\prime} _0 (x)/ \int_{0}^{\tau_0} dP^{\prime} _0 (x)dx$ $\forall x\in [0,\tau_0]$, and  $P^{\prime\prime}_1 (x)= P^{\prime}_1 (x)/ \int_{0}^{\tau_1} dP^{\prime}_1 (x)dx $ $\forall x\in [0,\tau_1]$. In addition, for the consistency of the test against all alternatives as we will see later we must impose that $\tau_0= \tau_1 $ in the case that the  support of the distribution functions $P_0$ and $P_1$ is not contained in the intervals $ [0,\tau_0] $ and $[0,\tau_1] $ respectively.

	
	This leads to consider the $U$ statistics under right censored suggested in \cite {bose1999strong} and apply the aforementioned standardization for multisample $U$ statistic under right censoring \cite{stute1993multi}. The corresponding statistics are the following:

	\begin{align}
	\hat{\epsilon}_{\alpha}(P_0,P_1)= 2  \frac{\sum_{i=1}^{n_0}  \sum_{j=1}^{n_1} W^{0}_{i:n_0} W^{1}_{i:n_1} ||X_{0i}-X_{1j}||^{\alpha}}{\sum_{i=1}^{n_0}  \sum_{j\neq i}^{n_1} W^{0}_{i:n_0} W^{1}_{i:n_1}} - \frac{\sum_{i=1}^{n_0} \sum_{j\neq i}^{n_0}  W^{0}_{i:n_0} W^{0}_{j:n_0}  ||X_{0i}-X_{0j}||^{\alpha}}{\sum_{i=1}^{n_0} \sum_{j\neq i}^{n_0} W^{0}_{i:n_0} W^{0}_{j:n_0}} \\
	- \frac{\sum_{i=1}^{n_1}  \sum_{i\neq j}^{n_1} W^{1}_{i:n_1} W^{1}_{j:n_1}  ||X_{1i}-X_{1j}||^{\alpha}}{\sum_{i=1}^{n_1}  \sum_{j\neq i}^{n_1} W^{1}_{i:n_1} W^{1}_{j:n_1}} \nonumber
	\end{align}
	\textbf{($U$ statistic energy distance under right censoring)},
	\begin{align}
	\hat{\gamma}_K^{2}(P_0,P_1)=  \frac{\sum_{i=1}^{n_0}  \sum_{j\neq i}^{n_0} W^{0}_{i:n_0} W^{0}_{j:n_0} K(X_{0i},X_{0j})}{ \sum_{j=1}^{n_0} \sum_{j\neq i}^{n_0} W^{0}_{i:n_0} W^{0}_{j:n_0}}+ \frac{\sum_{i=1}^{n_1}  \sum_{j\neq i}^{n_1} W^{1}_{i:n_1} W^{1}_{j:n_1}  K(X_{1i},X_{1j})}{\sum_{i=1}^{n_1}  \sum_{j\neq i }^{n_1}W^{1}_{i:n_1} W^{1}_{j:n_1}}\\-2\frac{\sum_{i=1}^{n_0}  \sum_{j=1}^{n_1} W^{0}_{i:n_0} W^{1}_{i:n_1} K(X_{0i},X_{1j})}{\sum_{i=1}^{n_0}  \sum_{j=1}^{n_1} W^{0}_{i:n_0} W^{1}_{i:n_1}} \nonumber
	\end{align}
	\textbf{($U$ statistic kernel method under right censoring)}.

	Finally, we will use the following statistics
	
	\begin{equation}
	T_{\hat{\epsilon}_\alpha}= \frac{n_0n_1}{n_0+n_1} \hat{\epsilon}_{\alpha}(P_0,P_1) \text{   and   }    T_{\hat{\gamma}_K^{2}}= \frac{n_0n_1}{n_0+n_1} \hat{\gamma}_K^{2}(P_0,P_1) 
	\end{equation}

	to derive more easily, the consistency against all alternatives.

	\subsection{Permutation tests}

	As in the case without censorship, the null distribution of the statistics is calculate with  permutation methods. If the censorship mechanism of the two groups is the same,  the standard permutation methods  are valid  \cite{neuhaus1993conditional,wang2010testing}. However, when the censoring distributions differ, standard permutation methods do not work well for small-sample settings and/or when the amount of censoring is large \cite{heimann1998permutational}. In this case, we must use the resampling strategy proposed in  \cite{wang2010testing}.
	

	We denote by $Z=(\overbrace{0,\cdots,0}^{n_0},\overbrace{1,\cdots,1}^{n_1})$ a vector the size $n$ $(n=n_0+n_1)$ that contains the group to which it belongs to each data, and by $U= (X_{01},\cdots, X_{0n_0},X_{11},\cdots ,X_{1n_1})$ and $\delta=(\delta_{01},\cdots ,\delta_{0n_0},\delta_{11},\cdots ,\delta_{1n_1})$ to the vectors of the same length that contain the observed times and the censorship indicator of each time. Given a statistic $\theta (Z, U,\delta)$, the first step of  traditional permutations method consists in calculate the value  of each statistics for each permutation $\theta(Z^r,U,\delta)$ $(r = 1,2, \dots, \binom{n}{n_0})$.  Resulting each permutation of consider  $\binom{n}{n_0}$ combination over the  index $\{1,\dots,n\}$    in the following way:  the values  of $\binom{n}{n_0})$ different possible combinations are distributed to the    first  group  and assigned  the $n-n_0$ remaining index to the other group. Finally, we compare if $\theta(Z,U,\delta)$ is less or equal that $\theta(Z^r,U,\delta)$ $(r = 1,2, \dots, \binom{n}{n_0})$. The p-value is calculated  as follow:
	
	\begin{equation*}
	\text{p-value}=\frac{\sum_{r=1}^{\binom{n}{n_0}}1\{{\theta(Z^r,U,\delta)\geq \theta(Z,U,\delta)\}}}{\binom{n}{n_0}}.
	\end{equation*}
	
	In practice, only a a small number of permutations is considered in the approximation of the latest expression.

	%
	%
	%

	\subsection{Theoretical properties}

	\subsubsection{Asymptotic  distribution}

	The theoretical results derived for the asymptotic convergence in distribution under null hypothesis of the statistics will be established only in the proofs for the case of kernel mean embeddings. As we have seen before (equation $(11)$) given the equivalence between the tests based on the kermel mean embeddings and the energy distance \cite{sejdinovic2013equivalence}  this is not restrictive.
	
	%

	We first transform each term in the previously sum by centering. Under the null hypothesis $P=P_0=P_1$ and $\tau_0=\tau_1$, $P^{\prime\prime}=P_0^{\prime\prime}=P_1^{\prime\prime}$ and we have the same mean embedding 
	
	$\mu_{P^{\prime\prime}}=\mu_{P_0^{\prime\prime}}=\mu_{P_1^{\prime\prime}}= \frac{1}{P(\tau_0)}\int_0^{\tau_0} K(\cdot,x) dP^{\prime \prime}(x) $. Thus if we replace each instance of $K(X_i,X_j)$ with a kernel $K^{*}(X_i,X_j)$ which the mean has been subtracted, 
	
	\begin{align*}
	K^{*}(X_i,X_j)= <\phi(X_i)-\mu_{P^{\prime\prime}},\phi(X_j)-\mu_{P^{\prime\prime}}>=  \\ K(X_i,X_j)-E_{X}(K(X_i.X))  -E_{X}(K(X,X_j))+E_{X,X^{\prime}}(K(X,X^{\prime})). 
	\end{align*}

	This gives the equivalent of the empirical $\hat{\gamma}_K^{2}(P_0^{\prime \prime},P_1^{\prime \prime})$

	\begin{align}
	\hat{\gamma}_K^{2}(P_0^{\prime\prime },P_1^{\prime\prime })=  \frac{\sum_{i=1}^{n_0}  \sum_{j\neq i}^{n_0} W^{0}_{i:n_0} W^{0}_{j:n_0} K^{*}(X_{0i},X_{0j})}{ \sum_{j=1}^{n_0} \sum_{j\neq i}^{n_0} W^{0}_{i:n_0} W^{0}_{j:n_0}}+ \frac{\sum_{i=1}^{n_1}  \sum_{j\neq i}^{n_1} W^{1}_{i:n_1} W^{1}_{j:n_1}  K^{*}(X_{1i},X_{1j})}{\sum_{i=1}^{n_1}  \sum_{j\neq i }^{n_1}W^{1}_{i:n_1} W^{1}_{j:n_1}} \nonumber\\-2\frac{\sum_{i=1}^{n_0}  \sum_{j=1}^{n_1} W^{0}_{i:n_0} W^{1}_{i:n_1} K^{*}(X_{0i},X_{1j})}{\sum_{i=1}^{n_0}  \sum_{j=1}^{n_1} W^{0}_{i:n_0} W^{1}_{i:n_1}}.
	\end{align}

	Note  that $K^{*}(\cdot,\cdot)$ is a degenerate kernel:
	
	$$ E_{X\distas{}P^{\prime\prime}}(K^{*}(X.y))= E_{X}(K(X,y)) + E_{X,X^{\prime}} K(X,X^{\prime}) -E_{X}(K(X,y))+E_{X,X^{\prime}} K(X,X^{\prime})=0 $$

	Then, in the terms 
	
	\begin{equation*}
	\frac{\sum_{i=1}^{n_0}  \sum_{j\neq i}^{n_0} W^{0}_{i:n_0} W^{0}_{j:n_0} K^{*}(X_{0i},X_{0j})}{ \sum_{j=1}^{n_0} \sum_{j\neq i}^{n_0} W^{0}_{i:n_0} W^{0}_{j:n_0}} \text{ and } \frac{\sum_{i=1}^{n_1}  \sum_{j\neq i}^{n_1} W^{1}_{i:n_1} W^{1}_{j:n_1}  K^{*}(X_{1i},X_{1j})}{\sum_{i=1}^{n_1}  \sum_{j\neq i }^{n_1}W^{1}_{i:n_1} W^{1}_{j:n_1}} \nonumber
	\end{equation*}

	we can apply the limits theorems for $U$ statistics under right censored data \cite{bose2002asymptotic,fernandez2018kaplan}. In particular we will use the results \cite{fernandez2018kaplan} due to the weakest conditions to apply the theorems, and also, for the conditions that are assumed in this work
	it is proved in that same work that the theorems of asymptotic convergence are valid.

	By the Corollary $2.9$ \cite{fernandez2018kaplan}, under the null hyphotesis and $\tau_0= \tau_1$  we have:

	$$\frac{\sum_{i=1}^{n_0}  \sum_{j\neq i}^{n_0} W^{0}_{i:n_0} W^{0}_{j:n_0} K^{*}(X_{0i},X_{0j})}{ \sum_{j=1}^{n_0} \sum_{j\neq i}^{n_0} (W^{0}_{i:n_0} W^{0}_{j:n_0})}  \overset{D}{\to} c_1 + \psi$$

	and

	$$\frac{\sum_{i=1}^{n_1}  \sum_{j\neq i}^{n_1} W^{1}_{i:n_1} W^{0}_{j:n_1} K^{*}(X_{1i},X_{1j})}{ \sum_{j=1}^{n_1} \sum_{j\neq i}^{n_1} (W^{1}_{i:n_1} W^{1}_{j:n_1}})  \overset{D}{\to} c_2 + \psi$$

	where $\psi= \sum_{i=1}^{\infty} \lambda_i (\epsilon_i^{2}-1)$, with $\epsilon_{i}$ $iid$ standard normal random variables and $c_1$, $c_2$ are two constant specified in \cite{fernandez2018kaplan} that for our purpose are not irrelevant.

	The structure of the previous limits coincides with the case without censoring  in the degenerate case corresponds to $c+\psi$ \cite{korolyuk2013theory} where $c$ is a constant.

	However, for the term

	$$2\frac{\sum_{i=1}^{n_0}  \sum_{j=1}^{n_1} W^{0}_{i:n_0} W^{1}_{i:n_1} K(X_{0i},X_{1j})}{\sum_{i=1}^{n_0}  \sum_{j=1}^{n_1} W^{0}_{i:n_0} W^{1}_{i:n_1}} $$

	which is a U-statistic of two samples under right censored data there are still no theoretical results.

	The deduction of the theorems limits with $ U $ statistics in several samples extends the objectives of this work, and will be presented in another paper. In any case, the limit distribution coincides with the case with censorship. This is

	$$ \sqrt{n_0n_1}  \sum_{i=1}^{n_0}  \sum_{j=1}^{n_1} W^{0}_{i:n_0} W^{1}_{i:n_1} K(X_{0i},X_{1j}) \overset{D}{\to} \eta_{\infty}$$
	
	$$\eta_{\infty}=\sum_{j=1}^{\infty} \lambda_j \tau_j \epsilon_j,$$
	
	where $\{\tau_j\}$ and $\{\epsilon_j\}$ are two independence sequences of standart normal random variables.
	
	\subsubsection{Consistency against all alternatives}

	\begin{theorem}
		Let $S,A$ be an arbitrary metrics spaces with the same topology defined on $\mathbb{R^{+}}$ with $S$ contained on $A$ and let $\gamma(x,y)$ be a continuous, symmetric, real function on  $A\times A$. Suppose $X$,$X^{\prime}$, $Y$,$Y^{\prime}$ are independent $A$ random variables, $X$,$X^{\prime}$ and identically distributed, and $Y$,$Y^{\prime}$ are identically distributed. Suppose 
		$\gamma(X,X^{\prime})$, $\gamma(Y,Y^{\prime})$, and $\gamma(X,Y)$ have finite expected values on $A$. Then
		
		\begin{equation*}
		2\frac{\int_S \int_S \gamma(x,y)dP(x) dQ(y)}{\int_{S}dP(x)\int_{S}dQ(y)}-	\frac{\int_S \int_S \gamma(x,y)dP(x) dP(y)}{(\int_{S}dP(x))^{2}}-\frac{\int_S \int_S \gamma(x,y)dQ(x) dQ(y)}{(\int_{S}dQ(x))^{2}}\geq 0
		\end{equation*}
		if and only if $\phi$ is negative definite and where $P$ and $Q$ denote the distribution of $X$ and $Y$ respectively. If $\gamma$ is strictly negative then equality holds if and only if $X$ and $Y$ are identically distributed on $S$.   
	\end{theorem}
	
	\begin{proof}
		%
		%
		
		By Theorem $1$ \cite{szekely2005new}, it is verified:

		\begin{equation}
		2\int_A \int_A \gamma(x,y)dP(x) dQ(y)-	\int_A	 \int_A \gamma(x,y)dP(x) dP(y)-\int_A \int_A \gamma(x,y)dQ(x) dQ(y)\geq 0,
		\end{equation}

		if and only if $\phi$ is negative definite. If $\gamma$ is strictly negative then equality holds if and only if $X$ and $Y$ are identically distributed on $A$.
		
		If we define the following random variables on $S$, $X^{*}$, $Y^{*}$ with distribution function $P^{\prime}$, $Q^{\prime}$ respectively as follow :
		
		$dP^{\prime}(x)= c_1 dP(x)$ and $dQ^{\prime}(x)= c_2 dP(x)$, where $c_1= \frac{1}{\int_{S} dP(x)}$ and $c_2= {\frac{1}{\int_{S} dQ(x)}}$, 
		And we consider his copies $X^{*\prime}$,$Y^{*\prime}$. As $\gamma(X,X^{\prime})$, $\gamma(Y,Y^{\prime})$, and $\gamma(X,Y)$ have finite expected values on $A$, then $\gamma(X^{*},X^{*\prime})$, $\gamma(Y^{*},Y^{*\prime})$, and $\gamma(X^{*},Y^{*})$ have finite expected values on $S$. Moreover, $\gamma(x,y)$ be a continuous, symmetric, real function on  $S\times S$.

		This leads:
		
		\begin{equation}
		2c_1c_2\int_S \int_S \gamma(x,y)dP(x) dQ(y)-	c_1^{2}\int_S \int_S \gamma(x,y)dP(x) dP(y)-c_2^{2}\int_S \int_S \gamma(x,y)dQ(x) dQ(y)\geq 0.
		\end{equation}
		
		if and only if $\phi$ is negative definite,  and 
		
		\begin{equation}
		2c_1c_2\int_S \int_S \gamma(x,y)dP(x) dQ(y)-	c_1^{2}\int_S \int_S \gamma(x,y)dP(x) dP(y)-c_2^{2}\int_S \int_S \gamma(x,y)dQ(x) dQ(y)= 0.
		\end{equation}

		if $X^{*}$ and $Y^{*}$ are identically distributed on $A$ (with $\phi$  strictly negative) or equivalent $X$ and $Y$ are equally distributed on $S$.

	\end{proof}

	\begin{theorem}
		
		Let $X_{ji}= min(T_{ji},C_{j,i})\distas{iid} P_{c(j)} $ and  $\delta_{ji}= 1\{X_{ji}= T_{ji}\}$	$(j=0,1; i=1,\dots,n_j)$ with $P_{c(j)}$ $(j=0,1)$ and under the conditions of  assumed in the section $2$ on the variables $T_{ji}\distas{iid} P_{j},C_{j,i}\distas{iid} Q_{j}$ $(j=0,1; i=1,\dots,n_j)$. Then:

		\begin{eqnarray}
		\hat{\epsilon}_\alpha(P_0,P_1) \overset{n_0,n_1\to \infty}{\to} \epsilon_{c(\alpha)}(P_0,P_1)= 2\frac{\int_{0}^{\tau_0} \int_{0}^{\tau_1}	 ||x-y||^{\alpha} dP_0^{\prime}(x) dP_1^{\prime}(y)}{\int_{0}^{\tau_0} \int_{0}^{\tau_1} dP_0^{\prime}(x) dP_1^{\prime}(y)}\\
		-\frac{\int_{0}^{\tau_0} \int_{0}^{\tau_0}	 ||x-y||^{\alpha} dP_0^{\prime}(x) dP_0^{\prime}(y)}{\int_{0}^{\tau_0} \int_{0}^{\tau_0} dP_0^{\prime}(x) dP_0^{\prime}(y)} -\frac{\int_{0}^{\tau_1}\int_{0}^{\tau_1}	 ||x-y||^{\alpha} dP_1^{\prime}(x) dP_1^{\prime}(y)}{\int_{0}^{\tau_1}\int_{0}^{\tau_1} dP_1^{\prime}(x) dP_1^{\prime}(y)},
		\end{eqnarray} 
		
		\begin{eqnarray}
		\hat{\gamma}_K(P_0,P_1) \overset{n_0,n_1\to \infty}{\to} \gamma_{c(K)}(P_0,P_1)= 2  \frac{\int_{0}^{\tau_0} \int_{0}^{\tau_1}	 K(x,y) dP_0^{\prime}(x) dP_1^{\prime}(y)}{\int_{0}^{\tau_0} \int_{0}^{\tau_1} dP_0^{\prime}(x) dP_1^{\prime}(y)}\\
		-\frac{\int_{0}^{\tau_0} \int_{0}^{\tau_0}	 K(x,y) dP_0^{\prime}(x) dP_0^{\prime}(y)}{\int_{0}^{\tau_0} \int_{0}^{\tau_0}	  dP_0^{\prime}(x) dP_0^{\prime}(y)}-\frac{\int_{0}^{\tau_1}\int_{0}^{\tau_1}	K(x,y) dP_1^{\prime}(x) dP_1^{\prime}(y)}{\int_{0}^{\tau_1}\int_{0}^{\tau_1} dP_1^{\prime}(x) dP_1^{\prime}(y)},
		\end{eqnarray}

		where
		
		$P_0^{\prime}(x)= \left\{ \begin{array}{lcc}
		P_0(x) &   if  & x < \tau_{0} \\
		\\ P_0(\tau_0^{-})+1\{\tau_0\in A^{1}\}P_0(\tau_0) &  if & x \geq \tau_0 \\
		\end{array}
		\right.$
		
		and
		
		$P_1^{\prime}(x)= \left\{ \begin{array}{lcc}
		P_1(x) &   if  & x < \tau_{1} \\
		\\ P_1(\tau_1^{-})+1\{\tau_1\in A^{1}\}P_1(\tau_1) &  if & x \geq \tau_1. \\
		
		\end{array}
		\right.$
		
		Here,
		$\tau_0=\inf\{x:1-P_{c(0)}(x)=0\}$, $\tau_1=\inf\{x:1-P_{c(1)}(x)=0\}$, $A^{0}=\{x\in \mathbb{R}| P_{c(0)}\{x\}>0\}$ and $A^{1}=\{x\in \mathbb{R}| P_{c(1)}\{x\}>0\}$. 
	\end{theorem}
	\begin{proof}
		
		%

		The proof consists of repeatedly applying the strong laws of large numbers for $U$ Kaplan Meier statistics with two samples \cite{stute1993multi}, with the convergence results for $U$ statistic of degree two for randomly censored \cite{bose1999strong}. 
		
		By \cite{stute1993multi} we know that

		$$ \sum_{i=1}^{n_1} \sum_{j=1}^{n_1} W^{0}_{i:n_0} W^{1}_{i:n_1} h(X_{0i},X_{1j}) \overset{n_0,n_1\to \infty}{\to} \int_{0}^{\tau_0} \int_{0}^{\tau_1} h(x,y) dP_0^{\prime}(x) dP_1^{\prime}(y) $$
		
		where $h$ is a given kernel of degree two such that
		
		$$ \int h(x,y) dP_0^{\prime}(x) dP_1^{\prime}(y)< \infty $$

		Note that by hypothesis that $ P_{c(j)}$ $(j = 0,1) $ is continuous distribution function implies that $A^{0}$ and $A^{1}$ are empty set and therefore $P^{\prime}_0(x)= P_0(x)$ $\forall \in[0,\tau_0]$  and $P^{\prime}_1(x)= P_1(x)$ $\forall\in[0,\tau_1]$

		Applying the previous result with $h(x,y)=1$, along with the properties of convergence in probability,  we have:
		
		$$    \frac{\sum_{i=1}^{n_1} \sum_{j=1}^{n_1} W^{0}_{i:n_0} W^{1}_{i:n_1} h(X_{0i},X_{1j})}{\sum_{i=1}^{n_1} \sum_{j=1}^{n_1} W^{0}_{i:n_0} W^{1}_{i:n_1}} \overset{n_0,n_1\to \infty}{\to} \frac{\int_{0}^{\tau_0} \int_{0}^{\tau_1} h(x,y) dP_0^{\prime}(x) dP_1^{\prime}(y)}{\int_{0}^{\tau_0} \int_{0}^{\tau_1}  dP_0^{\prime}(x) dP_1^{\prime}(y)} $$

		Using the theorem $1$ of \cite{bose1999strong},  it is verified also

		$$\frac{\sum_{i=1}^{n_0}  \sum_{j\neq i}^{n_0} W^{0}_{i:n_0} W^{0}_{j:n_0} h(x,y)}{\sum_{i=1}^{n_0}  \sum_{j\neq i}^{n_0} W^{0}_{i:n_0} W^{0}_{j:n_0}} \overset{n_0 \to \infty}{\to} \frac{\int_{0}^{\tau_0} \int_{0}^{\tau_0} h(x,y)	  dP_0^{\prime}(x) dP_0^{\prime}(y)}{\int_{0}^{\tau_0} \int_{0}^{\tau_0} dP_0^{\prime}(x) dP_0^{\prime}(y)} $$

		and

		$$\frac{\sum_{i=1}^{n_1}  \sum_{j\neq i}^{n_1} W^{1}_{i:n_1} W^{1}_{j:n_1} h(x,y)}{\sum_{i=1}^{n_1}  \sum_{j\neq i}^{n_1} W^{1}_{i:n_1} W^{1}_{j:n_1}} \overset{n_1 \to \infty}{\to} \frac{\int_{0}^{\tau_1} \int_{0}^{\tau_1} h(x,y)	  dP_1^{\prime}(x) dP_1^{\prime}(y)}{\int_{0}^{\tau_1} \int_{0}^{\tau_1} dP_1^{\prime}(x) dP_1^{\prime}(y)} $$.
		
		Finally taking as $h(x,y)= ||x-y||^{\alpha}$ or $h(x,y)=K(x,y)$ and applying the properties of convergence   in probability of the sum of two random variables, the desired result is obtained.

	\end{proof}

	\begin{theorem}
		Let $X_{ji}= min(T_{ji},C_{j,i})\distas{iid} P_{c(j)} $ and  $\delta_{ji}= 1\{X_{ji}= T_{ji}\}$	$(j=0,1; i=1,\dots,n_j)$ with $P_{c(j)}$ $(j=0,1)$ under the conditions of independence assumed in the section $2$ on the variables $T_{ji}\distas{iid} P_{j},C_{j,i}\distas{iid} Q_{j}$ $(j=0,1; i=1,\dots,n_j)$. Also let's suppose that $\tau_0=\tau_1$ or the  support of the distribution functions $P_0$ and $P_1$ is  contained in the intervals $ [0,\tau_0] $ and $[0,\tau_1] $ respectively. Then, the statistics $T_{\hat{\epsilon}_\alpha}$   $T_{\hat{\gamma}_K^{2}}$    determines a test of the hypothesis of equal distributions  that is consistent against all fixed  alternatives with continuos random variables.

		\begin{proof}

			We assume without any restriction that $P_0$ and $ P_1$ have the same support (otherwise it is enough to extend the probability measure with less support to the higher one). If $\tau_0= \tau_1$ we can  apply theorem $1$ and then we have guaranteed:

			\begin{align}
			\lim_{n_0\to \infty,n_1 \to \infty} \hat{\epsilon}_\alpha(P_0,P_1)= 
			2\frac{\int_{0}^{\tau_0} \int_{0}^{\tau_1}	 ||x-y||^{\alpha} dP_0^{\prime}(x) dP_1^{\prime}(y)}{\int_{0}^{\tau_0} \int_{0}^{\tau_1} dP_0^{\prime}(x) dP_1^{\prime}(y)}\\
			-\frac{\int_{0}^{\tau_0} \int_{0}^{\tau_0}	 ||x-y||^{\alpha} dP_0^{\prime}(x) dP_0^{\prime}(y)}{\int_{0}^{\tau_0} \int_{0}^{\tau_0} dP_0^{\prime}(x) dP_0^{\prime}(y)} -\frac{\int_{0}^{\tau_1}\int_{0}^{\tau_1}	 ||x-y||^{\alpha} dP_1^{\prime}(x) dP_1^{\prime}(y)}{\int_{0}^{\tau_1}\int_{0}^{\tau_1} dP_1^{\prime}(x) dP_1^{\prime}(y)} \geq 0 \nonumber
			\end{align}
			
			%
			%
			
			\begin{align}
			\lim_{n_0\to \infty,n_1 \to \infty} \hat{\gamma}_K(P_0,P_1)= \frac{\int_{0}^{\tau_0} \int_{0}^{\tau_0}	 K(x,y) dP_0^{\prime}(x) dP_0^{\prime}(y)}{\int_{0}^{\tau_0} \int_{0}^{\tau_0}	  dP_0^{\prime}(x) dP_0^{\prime}(y)} \\
			+\frac{\int_{0}^{\tau_1}\int_{0}^{\tau_1}	K(x,y) dP_1^{\prime}(x) dP_1^{\prime}(y)}{\int_{0}^{\tau_1}\int_{0}^{\tau_1} dP_1^{\prime}(x) dP_1^{\prime}(y)}  -2   \frac{\int_{0}^{\tau_0} \int_{0}^{\tau_1}	 K(x,y) dP_0^{\prime}(x) dP_1^{\prime}(y)}{\int_{0}^{\tau_0} \int_{0}^{\tau_1} dP_0^{\prime}(x) dP_1^{\prime}(y)} \geq 0  \nonumber
			\end{align}

			%
			%
			%
			%
			
			and giving the equality to zero if and only if $P_0(t)= P_1(t)$ $\forall t \in [0,t_1]$
			
			%


			Suppose $\exists t \in [0,\tau_1]$ $P_0 (t) \neq P_1 (t)$, then we have strictly	inequality in $(36,37)$, so  with probability one  $\lim_{n_0\to \infty,n_1 \to \infty} P(\hat{\epsilon}_{\alpha}(P_0,P_1) = c_{\epsilon_\alpha}>0)=1$ $\lim_{n_0\to \infty,n_1 \to \infty} P(\hat{\gamma}_K(P_0,P_1)= c_K>0)=1$.  
			By the theory of degenerate $U$-statistics  under the null hyphotesis there exists a constants $c_{\alpha_1}$ and  $c_{\alpha_2}$  satisfying

			\begin{equation*}
			\lim_{n\to \infty} P(\frac{n_0 n_1}{n_{0}+n_{1}} \hat{\epsilon}_{\alpha}(P_0,P_1) > c_{\alpha_1})= \alpha  \text{ and } \lim_{n\to \infty} P(\frac{n_0n_1}{n_{0}+n_1} \hat{\gamma}_K(P_0,P_1) > c_{\alpha_2})=\alpha.
			\end{equation*} 
			
			Under the alternative hypothesis
			
			\begin{equation*}
			\lim_{n\to \infty} P(\frac{n_0 n_1}{n_{0}+n_{1}} \hat{\epsilon}_{\alpha}(P_0,P_1) > c_{\alpha_1})=1  \text{ and } \lim_{n\to \infty} P(\frac{n_0n_1}{n_{0}+n_1} \hat{\gamma}_K(P_0,P_1) > c_{\alpha_2})=1.
			\end{equation*}

			since $n\hat{\epsilon}_{\alpha}(P_0,P_1) \to \infty$ and $ n\hat{\gamma}_K(P_0,P_1)$ with probabiliy one as $n\to \infty$.

			%

			In the case $\tau_0 \neq \tau_1$ the  support of the distribution functions $P_0$ and $P_1$ is  contained in the intervals $ [0,\tau_0] $ and $[0,\tau_1]$ and in this situation the normalization constants are $1$, and then, the previous argument is going to be true.

		\end{proof}    
		
	\end{theorem}


	\section{Simulation study}
	
	The simulation study is divided into two phases. In the first, the performance of the new tests proposed under the null hypothesis is compared with the logrank family tests with different censorship rates and different sample size. In particular, the tests used are the energy distance (with $\alpha \in\{0.4,0.8,1,1,2,1.6\} $), gaussian kernel $ (\sigma=1$), laplacian kernel $(\sigma=1$), rational quadratic ($c=\beta=1 $ and $c=\beta=2 $),
	log-rank, Gehan generalized Wilcoxon test, Tarone-Ware, Peto-Peto, Fleming $\&$ Harrington (with $\rho=\gamma=1$). For this purpose, parametric distributions such as normal, exponential or lognormal are used. In the second phase, the same tests are compared where the null hypothesis is not true, in different scenarios: proportional  hazard ratio, cure, multimodality, and delayed effect.We use    different censorship mechanisms for each case and we vary the sample size $n$ ($n\in\{20,50,100\}$).
	%
	
	All the tests  are executed on the statistical software R. For the family of the logrank test the coin package \cite{hothorn2008implementing} is  used, while the new tests have been implemented in C++, and integrating them in R with the Rcpp \cite{eddelbuettel2011rcpp}, and Rcpp Armadillo libraries. In both cases the tests are calibrated by the permutations method, performing $1000$ repetitions for our tests.
	
	\subsection{Null hyphotesis}
	

	We simulate $500$ times two samples in which the null hypothesis is correct. The censoring rates are $10$ and $30$ percent, and the sample size of $20$ and $50$ individuals. As under the null hypothesis 
	
	p-value    $\distas{}\text{Uniform}(0,1)$, the mean of the p-values obtained should be close to $0.5$, and the Standard deviation  $\sqrt{1/12}=0.2886751$. Likewise, approximately the $5$ percent of the observations should have a value less than $0.05$. In Table \ref{fig:n10}  we can see the results of calculating the mean and standard deviation for each test and case study contemplated, while in Table \ref{fig:n11} shows the proportion of p values that are less or equal than $0.05$ in the same cases.

\begin{sidewaystable}[h!] 
		\caption{Empirical mean and standard deviation of p-values for each case of study under the null hypothesis.}
		\centering
		\resizebox{25cm}{!}{%
			\begin{tabular}{cccccccccccccccccccc}			\hline
				Method: & &	& &	 Energy distance  & Energy distance & Energy distance & Energy distance &Energy distance & Kernel & Kernel & Kernel & Kernel & Logrank & Gehan & Tarone & Peto & Flemming \\
				& &	& &	 $\alpha=1$ & $\alpha=0.4$ & $\alpha=0.8$ & $\alpha=1.2$ & $\alpha=1.6$ & Gaussian $\sigma=1$  & $Laplacian$ $\sigma=1$ & Quadratic $c=1, \beta=1$ & Quadratic  $c=2, \beta=2$ & & & & & $\rho=1,\gamma=1$  \\			
				
				\\

				Comparative	& $n_1$ & $n_2$ & Censoring rate   & $\overline{x}\Mypm \sigma$ & $\overline{x}\Mypm \sigma$ & $\overline{x}\Mypm \sigma$ & $\overline{x}\Mypm \sigma$ & $\overline{x}\Mypm sd$ & $\overline{x}\Mypm \sigma$ & $\overline{x}\Mypm \sigma$ & $\overline{x}\Mypm \sigma$ & $\overline{x}\Mypm \sigma$ & $\overline{x}\Mypm \sigma$ & $\overline{x}\Mypm \sigma$ & $\overline{x}\Mypm \sigma$ & $\overline{x}\Mypm \sigma$ & $\overline{x}\Mypm \sigma$	
				
				\\ 
				\hline
				Exp(1)	& 20 & 20  & 0.1 & 0.496 $\Mypm$   0.286 & 0.493 $\Mypm$   0.283 & 0.495 $\Mypm$   0.285 & 0.497 $\Mypm$   0.287 & 0.499 $\Mypm$   0.289 & 0.495 $\Mypm$   0.287 & 0.493 $\Mypm$   0.284 & 0.493 $\Mypm$   0.288 & 0.497 $\Mypm$   0.286 & 0.508 $\Mypm$   0.291 & 0.505 $\Mypm$   0.288 & 0.503 $\Mypm$   0.285 & 0.502 $\Mypm$   0.285 & 0.495 $\Mypm$   0.296 \\ 
				Exp(1)	& 50 & 50 & 0.1 & 0.482 $\Mypm$   0.293 & 0.478 $\Mypm$   0.290 & 0.481 $\Mypm$   0.293 & 0.483 $\Mypm$   0.292 & 0.485 $\Mypm$   0.289 & 0.481 $\Mypm$   0.298 & 0.478 $\Mypm$   0.294 & 0.479 $\Mypm$   0.296 & 0.482 $\Mypm$   0.297 & 0.492 $\Mypm$   0.293 & 0.489 $\Mypm$   0.295 & 0.478 $\Mypm$   0.280 & 0.486 $\Mypm$   0.292 & 0.490 $\Mypm$   0.295 \\ 
				Exp(1.5)	& 20 & 20	 & 0.1 & 0.482 $\Mypm$   0.287 & 0.490 $\Mypm$   0.285 & 0.484 $\Mypm$   0.286 & 0.480 $\Mypm$   0.287 & 0.477 $\Mypm$   0.288 & 0.489 $\Mypm$   0.290 & 0.49 $\Mypm$   0.285 & 0.493 $\Mypm$   0.288 & 0.483 $\Mypm$   0.289 & 0.471 $\Mypm$   0.296 & 0.486 $\Mypm$   0.289 & 0.475 $\Mypm$   0.291 & 0.481 $\Mypm$   0.288 & 0.458 $\Mypm$   0.288 \\ 
				Exp(1.5)	& 50 & 50	 & 0.1 & 0.482 $\Mypm$   0.295 & 0.475 $\Mypm$   0.288 & 0.481 $\Mypm$   0.293 & 0.483 $\Mypm$   0.296 & 0.484 $\Mypm$   0.297 & 0.485 $\Mypm$   0.293 & 0.481 $\Mypm$   0.289 & 0.487 $\Mypm$   0.289 & 0.483 $\Mypm$   0.295 & 0.492 $\Mypm$   0.299 & 0.501 $\Mypm$   0.295 & 0.492 $\Mypm$   0.295 & 0.498 $\Mypm$   0.295 & 0.479 $\Mypm$   0.293 \\ 
				Exp(1)	& 20 & 20	 & 0.3 & 0.508 $\Mypm$   0.288 & 0.509 $\Mypm$   0.288 & 0.508 $\Mypm$   0.288 & 0.507 $\Mypm$   0.288 & 0.508 $\Mypm$   0.290 & 0.503 $\Mypm$   0.285 & 0.506 $\Mypm$   0.287 & 0.502 $\Mypm$   0.286 & 0.504 $\Mypm$   0.287 & 0.495 $\Mypm$   0.284 & 0.502 $\Mypm$   0.297 & 0.500 $\Mypm$   0.294 & 0.499 $\Mypm$   0.295 & 0.507 $\Mypm$   0.286 \\ 
				Exp(1)	& 50 & 50	 & 0.3 & 0.494 $\Mypm$   0.297 & 0.496 $\Mypm$   0.295 & 0.494 $\Mypm$   0.297 & 0.494 $\Mypm$   0.295 & 0.494 $\Mypm$   0.291 & 0.493 $\Mypm$   0.297 & 0.495 $\Mypm$   0.297 & 0.494 $\Mypm$   0.298 & 0.493 $\Mypm$   0.297 & 0.503 $\Mypm$   0.296 & 0.486 $\Mypm$   0.297 & 0.491 $\Mypm$   0.296 & 0.486 $\Mypm$   0.297 & 0.502 $\Mypm$   0.287 \\ 
				Exp(1.5)	& 20 & 20	 & 0.3 & 0.500 $\Mypm$   0.290 & 0.510 $\Mypm$   0.293 & 0.503 $\Mypm$   0.291 & 0.498 $\Mypm$   0.289 & 0.495 $\Mypm$   0.288 & 0.492 $\Mypm$   0.284 & 0.506 $\Mypm$   0.292 & 0.498 $\Mypm$   0.288 & 0.492 $\Mypm$   0.284 & 0.497 $\Mypm$   0.295 & 0.499 $\Mypm$   0.289 & 0.493 $\Mypm$   0.285 & 0.495 $\Mypm$   0.286 & 0.501 $\Mypm$   0.292 \\ 
				Exp(1.5)	& 50 & 50	 & 0.3 & 0.489 $\Mypm$   0.301 & 0.487 $\Mypm$   0.297 & 0.488 $\Mypm$   0.300 & 0.489 $\Mypm$   0.301 & 0.490 $\Mypm$   0.299 & 0.489 $\Mypm$   0.301 & 0.486 $\Mypm$   0.299 & 0.486 $\Mypm$   0.301 & 0.490 $\Mypm$   0.302 & 0.496 $\Mypm$   0.298 & 0.492 $\Mypm$   0.294 & 0.495 $\Mypm$   0.299 & 0.492 $\Mypm$   0.294 & 0.500 $\Mypm$   0.299 \\ 
				Gamma(1,1)	& 20  & 20 & 0.1 & 0.501 $\Mypm$   0.294 & 0.508 $\Mypm$   0.297 & 0.503 $\Mypm$   0.295 & 0.499 $\Mypm$   0.293 & 0.493 $\Mypm$   0.288 & 0.512 $\Mypm$   0.297 & 0.508 $\Mypm$   0.296 & 0.511 $\Mypm$   0.296 & 0.505 $\Mypm$   0.295 & 0.491 $\Mypm$   0.284 & 0.510 $\Mypm$   0.294 & 0.498 $\Mypm$   0.288 & 0.506 $\Mypm$   0.292 & 0.493 $\Mypm$   0.282 \\ 
				Gamma(1,1)	& 50  & 50 & 0.1& 0.503 $\Mypm$   0.291 & 0.504 $\Mypm$   0.288 & 0.504 $\Mypm$   0.289 & 0.503 $\Mypm$   0.293 & 0.502 $\Mypm$   0.297 & 0.512 $\Mypm$   0.292 & 0.508 $\Mypm$   0.288 & 0.511 $\Mypm$   0.290 & 0.511 $\Mypm$   0.293 & 0.505 $\Mypm$   0.292 & 0.508 $\Mypm$   0.287 & 0.505 $\Mypm$   0.290 & 0.508 $\Mypm$   0.288 & 0.502 $\Mypm$   0.290 \\ 
				Gamma(1.5,1.5)	& 20  & 20 & 0.1 & 0.519 $\Mypm$   0.295 & 0.516 $\Mypm$   0.289 & 0.519 $\Mypm$   0.294 & 0.520 $\Mypm$   0.296 & 0.522 $\Mypm$   0.295 & 0.515 $\Mypm$   0.301 & 0.516 $\Mypm$   0.295 & 0.515 $\Mypm$   0.299 & 0.516 $\Mypm$   0.299 & 0.52 $\Mypm$   0.290 & 0.519 $\Mypm$   0.289 & 0.522 $\Mypm$   0.291 & 0.516 $\Mypm$   0.287 & 0.509 $\Mypm$   0.287 \\ 
				Gamma(1.5,1.5)	& 50  & 50 & 0.1&  0.499 $\Mypm$   0.290 & 0.493 $\Mypm$   0.291 & 0.497 $\Mypm$   0.290 & 0.501 $\Mypm$   0.289 & 0.506 $\Mypm$   0.287 & 0.494 $\Mypm$   0.289 & 0.495 $\Mypm$   0.292 & 0.493 $\Mypm$   0.291 & 0.498 $\Mypm$   0.289 & 0.515 $\Mypm$   0.295 & 0.505 $\Mypm$   0.289 & 0.509 $\Mypm$   0.289 & 0.506 $\Mypm$   0.288 & 0.505 $\Mypm$   0.291 \\ 
				Gamma(1,1)	& 20  & 20 & 0.3  & 0.477 $\Mypm$   0.288 & 0.485 $\Mypm$   0.288 & 0.479 $\Mypm$   0.288 & 0.475 $\Mypm$   0.289 & 0.474 $\Mypm$   0.289 & 0.479 $\Mypm$   0.289 & 0.484 $\Mypm$   0.287 & 0.484 $\Mypm$   0.288 & 0.475 $\Mypm$   0.289 & 0.477 $\Mypm$   0.297 & 0.467 $\Mypm$   0.288 & 0.464 $\Mypm$   0.288 & 0.463 $\Mypm$   0.287 & 0.489 $\Mypm$   0.292 \\ 
				Gamma(1,1)	& 50  & 50 & 0.3	 & 0.489 $\Mypm$   0.293 & 0.497 $\Mypm$   0.296 & 0.491 $\Mypm$   0.294 & 0.486 $\Mypm$   0.290 & 0.482 $\Mypm$   0.287 & 0.495 $\Mypm$   0.289 & 0.497 $\Mypm$   0.293 & 0.495 $\Mypm$   0.288 & 0.492 $\Mypm$   0.291 & 0.485 $\Mypm$   0.292 & 0.513 $\Mypm$   0.300 & 0.498 $\Mypm$   0.293 & 0.511 $\Mypm$   0.300 & 0.474 $\Mypm$   0.287 \\  
				Gamma(1.5,1.5)	& 20  & 20 & 0.3 & 0.491 $\Mypm$   0.293 & 0.494 $\Mypm$   0.293 & 0.492 $\Mypm$   0.294 & 0.491 $\Mypm$   0.293 & 0.489 $\Mypm$   0.293 & 0.493 $\Mypm$   0.294 & 0.494 $\Mypm$   0.294 & 0.494 $\Mypm$   0.295 & 0.492 $\Mypm$   0.293 & 0.484 $\Mypm$   0.297 & 0.499 $\Mypm$   0.294 & 0.490 $\Mypm$   0.295 & 0.494 $\Mypm$   0.292 & 0.484 $\Mypm$   0.294 \\ 
				Gamma(1.5,1.5)	& 50  & 50 & 0.3	& 0.495 $\Mypm$   0.295 & 0.489 $\Mypm$   0.294 & 0.493 $\Mypm$   0.294 & 0.496 $\Mypm$   0.295 & 0.499 $\Mypm$   0.293 & 0.492 $\Mypm$   0.293 & 0.49 $\Mypm$   0.295 & 0.491 $\Mypm$   0.294 & 0.492 $\Mypm$   0.294 & 0.509 $\Mypm$   0.289 & 0.49 $\Mypm$   0.291 & 0.493 $\Mypm$   0.288 & 0.489 $\Mypm$   0.292 & 0.514 $\Mypm$   0.288 \\ 
				Lognormal(0,0.5) & 20 &	20 & 0.1 & 0.49 $\Mypm$   0.287 & 0.495 $\Mypm$   0.288 & 0.492 $\Mypm$   0.287 & 0.489 $\Mypm$   0.288 & 0.488 $\Mypm$   0.290 & 0.49 $\Mypm$   0.283 & 0.493 $\Mypm$   0.287 & 0.489 $\Mypm$   0.284 & 0.490 $\Mypm$   0.285 & 0.472 $\Mypm$   0.279 & 0.477 $\Mypm$   0.287 & 0.470 $\Mypm$   0.285 & 0.473 $\Mypm$   0.286 & 0.483 $\Mypm$   0.287 \\ 
				Lognormal(0,0.5) & 50 &	50 & 0.10 &0.503 $\Mypm$   0.283 & 0.506 $\Mypm$   0.284 & 0.504 $\Mypm$   0.283 & 0.503 $\Mypm$   0.283 & 0.502 $\Mypm$   0.282 & 0.500 $\Mypm$   0.283 & 0.506 $\Mypm$   0.285 & 0.505 $\Mypm$   0.285 & 0.500 $\Mypm$   0.283 & 0.508 $\Mypm$   0.283 & 0.504 $\Mypm$   0.279 & 0.508 $\Mypm$   0.286 & 0.504 $\Mypm$   0.281 & 0.515 $\Mypm$   0.296 \\ 
				Lognormal(0,0.25) & 20 &	20 & 0.1 & 0.481 $\Mypm$   0.294 & 0.484 $\Mypm$   0.300 & 0.482 $\Mypm$   0.295 & 0.48 $\Mypm$   0.292 & 0.480 $\Mypm$   0.292 & 0.481 $\Mypm$   0.294 & 0.482 $\Mypm$   0.296 & 0.480 $\Mypm$   0.293 & 0.481 $\Mypm$   0.295 & 0.484 $\Mypm$   0.291 & 0.476 $\Mypm$   0.295 & 0.473 $\Mypm$   0.291 & 0.471 $\Mypm$   0.293 & 0.487 $\Mypm$   0.290 \\ 
				Lognormal(0,0.25) & 50 &	50 & 0.1  & 0.517 $\Mypm$   0.289 & 0.512 $\Mypm$   0.287 & 0.516 $\Mypm$   0.288 & 0.518 $\Mypm$   0.290 & 0.519 $\Mypm$   0.293 & 0.517 $\Mypm$   0.291 & 0.516 $\Mypm$   0.288 & 0.515 $\Mypm$   0.288 & 0.518 $\Mypm$   0.292 & 0.517 $\Mypm$   0.292 & 0.523 $\Mypm$   0.291 & 0.522 $\Mypm$   0.293 & 0.522 $\Mypm$   0.291 & 0.506 $\Mypm$   0.284 \\ 
				Lognormal(0,0.0.5) & 20 &	20 & 0.3  & 0.495 $\Mypm$   0.288 & 0.498 $\Mypm$   0.287 & 0.496 $\Mypm$   0.288 & 0.493 $\Mypm$   0.288 & 0.489 $\Mypm$   0.287 & 0.495 $\Mypm$   0.287 & 0.497 $\Mypm$   0.287 & 0.496 $\Mypm$   0.284 & 0.493 $\Mypm$   0.289 & 0.495 $\Mypm$   0.285 & 0.489 $\Mypm$   0.288 & 0.488 $\Mypm$   0.287 & 0.485 $\Mypm$   0.286 & 0.516 $\Mypm$   0.289 \\ 
				Lognormal(0,0.5) & 50 &	50 & 0.3  & 0.482 $\Mypm$   0.293 & 0.490 $\Mypm$   0.297 & 0.485 $\Mypm$   0.295 & 0.480 $\Mypm$   0.292 & 0.476 $\Mypm$   0.291 & 0.482 $\Mypm$   0.294 & 0.488 $\Mypm$   0.297 & 0.484 $\Mypm$   0.294 & 0.480 $\Mypm$   0.294 & 0.476 $\Mypm$   0.296 & 0.473 $\Mypm$   0.293 & 0.468 $\Mypm$   0.287 & 0.47 $\Mypm$   0.292 & 0.487 $\Mypm$   0.296 \\ 
				Lognormal(0,0.25) & 20 &	20 & 0.3 & 0.522 $\Mypm$   0.293 & 0.513 $\Mypm$   0.289 & 0.519 $\Mypm$   0.292 & 0.523 $\Mypm$   0.295 & 0.525 $\Mypm$   0.297 & 0.526 $\Mypm$   0.298 & 0.519 $\Mypm$   0.292 & 0.526 $\Mypm$   0.298 & 0.526 $\Mypm$   0.298 & 0.516 $\Mypm$   0.293 & 0.526 $\Mypm$   0.306 & 0.524 $\Mypm$   0.303 & 0.524 $\Mypm$   0.306 & 0.518 $\Mypm$   0.286 \\ 
				Lognormal(0,0.25) & 50 &	50 & 0.3  & 0.504 $\Mypm$   0.291 & 0.508 $\Mypm$   0.287 & 0.505 $\Mypm$   0.290 & 0.504 $\Mypm$   0.292 & 0.504 $\Mypm$   0.295 & 0.501 $\Mypm$   0.296 & 0.504 $\Mypm$   0.29 & 0.499 $\Mypm$   0.294 & 0.502 $\Mypm$   0.296 & 0.491 $\Mypm$   0.289 & 0.500 $\Mypm$   0.296 & 0.496 $\Mypm$   0.295 & 0.498 $\Mypm$   0.295 & 0.494 $\Mypm$   0.289 \\ 
				\hline
			\end{tabular}
				\label{fig:n10}
		}
	\end{sidewaystable} 

	\begin{sidewaystable}[ht]
		
		\caption{Proportion p-values less or equal $0.05$ for each case of study  under the null hypothesis.}
		\centering
		\resizebox{25cm}{!}{%
			\begin{tabular}{cccccccccccccccccccc}			\hline
				Method: & &	& &	 Energy distance  & Energy distance & Energy distance & Energy distance &Energy distance & Kernel & Kernel & Kernel & Kernel & Logrank & Gehan & Tarone & Peto & Flemming \\
				& &	& &	 $\alpha=1$ & $\alpha=0.4$ & $\alpha=0.8$ & $\alpha=1.2$ & $\alpha=1.6$ & Gaussian $\sigma=1$  & $Laplacian$ $\sigma=1$ & Quadratic $c=1, \beta=1$ & Quadratic  $c=2, \beta=2$ & & & & & $\rho=1,\gamma=1$  \\

				\\

				Comparative	& $n_1$ & $n_2$ & Censoring rate   & $\hat{p}$ & $\hat{p}$ & $\hat{p}$ & $\hat{p}$ & $\hat{p}$ & $\hat{p}$ & $\hat{p}$ & $\hat{p}$ & $\hat{p}$ & $\hat{p}$ & $\hat{p}$ & $\hat{p}$ & $\hat{p}$ & $\hat{p}$	
				
				\\

				\hline
				Exp(1) & 20 & 20 & 0.1 & 0.048 & 0.048 & 0.048 & 0.048 & 0.050 & 0.052 & 0.046 & 0.050 & 0.054 & 0.050 & 0.046 & 0.046 & 0.044 & 0.058 \\ 
				Exp(1) & 50 & 50 & 0.1 & 0.056 & 0.060 & 0.054 & 0.058 & 0.056 & 0.052 & 0.056 & 0.056 & 0.056 & 0.060 & 0.066 & 0.064 & 0.064 & 0.056 \\ 
				Exp(1.5) & 20 & 20  & 0.1 &  0.066 & 0.056 & 0.070 & 0.066 & 0.066 & 0.072 & 0.058 & 0.064 & 0.066 & 0.066 & 0.058 & 0.062 & 0.058 & 0.062 \\ 
				Exp(1.5) & 50 & 50 & 0.1 & 0.042 & 0.048 & 0.042 & 0.042 & 0.046 & 0.048 & 0.044 & 0.044 & 0.042 & 0.062 & 0.056 & 0.050 & 0.056 & 0.054 \\ 
				Exp(1) & 20 & 20 & 0.3 & 0.058 & 0.058 & 0.060 & 0.060 & 0.060 & 0.005 & 0.042 & 0.042 & 0.056 & 0.058 & 0.054 & 0.056 & 0.056 & 0.056 \\ 
				Exp(1) & 50 & 50 &  0.3 & 0.056 & 0.056 & 0.056 & 0.056 & 0.048 & 0.054 & 0.050 & 0.056 & 0.050 & 0.050 & 0.052 & 0.054 & 0.058 & 0.046 \\ 
				Exp(1.5) & 20 & 20 & 0.3 & 0.058 & 0.048 & 0.056 & 0.054 & 0.052 & 0.052 & 0.052 & 0.048 & 0.056 & 0.054 & 0.054 & 0.050 & 0.052 & 0.052 \\ 
				Exp(1.5) & 50 & 50 & 0.3 & 0.064 & 0.048 & 0.062 & 0.064 & 0.074 & 0.064 & 0.044 & 0.054 & 0.056 & 0.066 & 0.068 & 0.066 & 0.068 & 0.056 \\ 
				Gamma(1,1) & 20 & 20 & 0.3 & 0.058 & 0.056 & 0.060 & 0.056 & 0.054 & 0.052 & 0.060 & 0.060 & 0.052 & 0.054 & 0.054 & 0.056 & 0.058 & 0.054 \\ 
				Gamma(1,1) & 50 & 50 & 0.1 & 0.044 & 0.042 & 0.044 & 0.042 & 0.044 & 0.042 & 0.040 & 0.038 & 0.042 & 0.038 & 0.038 & 0.030 & 0.032 & 0.050 \\ 
				Gamma(1.5,1.5) & 20 & 20 & 0.1  & 0.062 & 0.058 & 0.062 & 0.066 & 0.066 & 0.060 & 0.060 & 0.056 & 0.064 & 0.046 & 0.048 & 0.048 & 0.048 & 0.062 \\ 
				Gamma(1.5,1.5) & 50 & 50 & 0.1 & 0.050 & 0.054 & 0.052 & 0.048 & 0.048 & 0.054 & 0.052 & 0.050 & 0.052 & 0.046 & 0.044 & 0.046 & 0.044 & 0.050 \\ 
				Gamma(1,1) & 20 & 20 & 0.3 & 0.058 & 0.058 & 0.062 & 0.058 & 0.058 & 0.064 & 0.054 & 0.060 & 0.062 & 0.056 & 0.060 & 0.058 & 0.052 & 0.066 \\ 
				Gamma(1,1) & 50 & 50 & 0.3 & 0.058 & 0.060 & 0.060 & 0.060 & 0.056 & 0.056 & 0.062 & 0.060 & 0.054 & 0.054 & 0.052 & 0.058 & 0.050 & 0.046 \\ 
				Gamma(1.5,1.5) & 20 & 20 & 0.3  & 0.068 & 0.056 & 0.066 & 0.068 & 0.066 & 0.070 & 0.056 & 0.060 & 0.070 & 0.058 & 0.060 & 0.064 & 0.062 & 0.066 \\ 
				Gamma(1.5,1.5) & 50 & 50 & 0.3 &  0.056 & 0.060 & 0.058 & 0.054 & 0.052 & 0.056 & 0.068 & 0.064 & 0.066 & 0.050 & 0.062 & 0.060 & 0.062 & 0.050 \\ 
				Lognormal(0,0.5) & 20 & 20 & 0.1 & 0.050 & 0.046 & 0.050 & 0.046 & 0.042 & 0.052 & 0.054 & 0.058 & 0.044 & 0.052 & 0.044 & 0.044 & 0.042 & 0.048 \\ 
				Lognormal(0,0.5) & 50 & 50 & 0.1 & 0.040 & 0.040 & 0.036 & 0.040 & 0.042 & 0.038 & 0.040 & 0.040 & 0.040 & 0.040 & 0.034 & 0.040 & 0.036 & 0.040 \\ 
				Lognormal(0,0.25) & 20 & 20 & 0.1 & 0.084 & 0.080 & 0.082 & 0.080 & 0.078 & 0.076 & 0.080 & 0.078 & 0.074 & 0.062 & 0.078 & 0.076 & 0.080 & 0.054 \\ 
				Lognormal(0,0.25) & 50 & 50 & 0.1 & 0.038 & 0.040 & 0.042 & 0.040 & 0.038 & 0.040 & 0.044 & 0.044 & 0.034 & 0.036 & 0.044 & 0.044 & 0.040 & 0.038 \\ 
				Lognormal(0,0.5) & 20 & 20 & 0.3 & 0.046 & 0.042 & 0.050 & 0.050 & 0.048 & 0.050 & 0.046 & 0.050 & 0.050 & 0.042 & 0.052 & 0.040& 0.048 & 0.050 \\ 
				Lognormal(0,0.5) & 50 & 50 & 0.3 & 0.072 & 0.076 & 0.074 & 0.076 & 0.076 & 0.074 & 0.074 & 0.072 & 0.072 & 0.078 & 0.082 & 0.078 & 0.082 & 0.066 \\ 
				Lognormal(0,0.25) & 20 & 20 & 0.3 & 0.056 & 0.052 & 0.054 & 0.056 & 0.060 & 0.060 & 0.054 & 0.056 & 0.062 & 0.050 & 0.056 & 0.058 & 0.054 & 0.042 \\ 
				Lognormal(0,0.25) & 50 & 50 & 0.3 & 0.044 & 0.050 & 0.046 & 0.046 & 0.040 & 0.040 & 0.052 & 0.044 & 0.040 & 0.046 & 0.060 & 0.046 & 0.058 & 0.048 \\ 
				\hline
		\end{tabular}} 
				\label{fig:n11}	
	
\end{sidewaystable}


	The results shown of the new tests proposed under the null hypothesis are consistent and similar to those of the logrank test family. Note that it is normal that there are certain discrepancies with the theoretical values when doing the comparison with $500$ repetitions, in $14$ different tests. In turn, the Kaplan-Meier estimator used in our models and in some of the logrank family presents a certain bias (dependent on the censoring ratio), which produces small deviations under what is expected in a theoretical framework under the null hypothesis.

	\subsection{Alternative hyphotesis}

	As before, we simulate $500$ repetitions of two samples, but this time the null hypothesis is unfulfilled. The cases we studied are the following: the hazard ratio is proportional between  two populations (the logrank test is the most powerful test in this context), healing occurs in a one population,  in a population the density function has several modes as a consequence of a multimodal treatment, there is  delayed effects in a population. The sample size vary by $20$, $50$ and $100$ people in each group and the censoring mechanics change between experiments. The significance level $\alpha$  of $0.05$ is used as the cutoff for significance.

	In each figure for each subcase we represent four graphs:  In the first one, the power of the tests of the energy of data, in the second of the kernel methods, in the third of the logrank test together with the other family methods, and in the last, the logrank test, the average power of the energy of data tests, of the kernel methods, and of the family logrank test.

	\subsubsection{Proportional hazard ratio in  two population}

	We simulate $500$ times varying the sample sizes with $20$ individuals from each group, $50$ and $100$, in  the following $10$ cases of study: $X\distas{}Exp(1)$ versus $Y\distas{}Exp(\theta) \text{ }(\text{with} \text{ }\theta\in\{1,1.1,1.3,1.4,1.5,1.6,1.7,1.8,1.9,2\})$.

	We representate the results based on the variation of the parameter $\theta$ for each sample size $n$ in different figures. In figure \ref{fig:n1} we show the results for $n = 20$, in figure \ref{fig:n2}, for $n =50$, and finally in figure \ref{fig:n3} for $n = 100$. As we can see in the three figures, the logrank test is usually the most powerful test, as is logical in the situation where this test is optimal from a theoretical point of view. However, the average of the results obtained by the distance of energy  is not far in statistical power. We can also appreciate that the selection of the parameters of both the energy distance and the kernel methods leads to more or less power for this case study,  which gives great flexibility to the family of tests.

	\subsubsection{Cure}
	
	We simulate data with the next predefined hazard ratio  function for each population on $[0,100]$: 
	
	$$
	\lambda_0(t)= 0.5 \text{  }  \forall t\in[0,100]
	$$
	
	and
	
	$$
	\lambda_1(t)= \left\{ \begin{array}{lcc}
	0.5-0.1t &   if  & 0\leq t \leq 5 \\
	
	0 &  if  & 5 <t \leq 100
	\end{array}.
	\right
	.
	$$

	The censoring times $C_{ij}\distas{iid} Uniform(0,10)$ $(j=0,1,i=1, \dots,n_j)$. In the figure \ref{fig:n4} we can see the graphical representation of the survival function resulting from calculating the Kaplan-Meier estimator using $ 5000 $ subjects of each group. Figure \ref{fig:n5} collects the results of the power study, where it can be seen that in this case the most powerful tests are those given by energy distance and kernel methods. It is curious that in the tests of these two families there is hardly any variability between the tests studied, however this is not the case in the family logrank test where there are many differences between the different tests.


	\subsubsection{Multimodality}

	We simulate data also with  default hazard ratio  function  for each population on $[0,100]$:

	$$
	\lambda_0(t)= 0.45 \text{  }  \forall t\in[0,100]
	$$
	
	and
	
	$$
	\lambda_1(t)= \left\{ \begin{array}{lcc}
	0.2+0.1t &   if  & 0\leq t \leq 1 \\
	
	0.5 &   if  & 1< t \leq 2 \\

	0.2+0.1(t-2) &   if  &2< t \leq 3 \\   
	0.5 & if  & 3<t\leq 4 \\

	0.2+0.1(t-4) &   if  & 4< t \leq 5 \\
	
	0.5 &   if  & 5< t \leq 100
	
	\end{array}.
	\right
	.
	$$

	The censoring times $C_{ij}\distas{iid} Uniform(0,10)$ $(j=0,1,i=1,\dots,n_j)$.  The figure \ref{fig:n6} show the Kaplan Meier estimator of each group. In this case of study, the family of the logrank test has more power  (figure \ref{fig:n7})  than the chosen tests based on energy distance and the kernel method. In turn, there is much discrepancy in the power achieved in many tests of this family, with some of them like Fleming having less power than the new tests proposed.

	\subsubsection{Delayed effect}
	
	We consider the next hazard ratio functions for each population:
	
	$$
	\lambda_0(t)= 0.4 \text{  }  \forall t\in[0,100]
	$$
	
	and
	
	$$
	\lambda_1(t)= \left\{ \begin{array}{lcc}
	0.6-0.1t &   if  & 0\leq t \leq 5 \\
	
	0.1 &   if  & 5 < t \leq 100 \\
	\end{array}.
	\right
	.
	$$

	The censoring variables $C_{ij}\distas{iid} Uniform(0,15)$ $(j=0,1,i=1,\dots,n_j)$. The Kaplan Meier estimator is shown in the figure \ref{fig:n8}.  
	
	In this last simulation, the methods based on kernel methods are the most powerful by far (figure \ref{fig:n9}). The power achieved by the log rank family tests and the energy distance is similar. However, the power of the log rank test is very low, with hardly any greater detection capacity than under the null hypothesis. In addition, this also occurs at the energy distance for $\alpha= 1.2$ and $\alpha= 1.6$, which shows that the appropriate parameter selection is necessary for the correct use of these tests.

	\section{Final remarks}
	
	In this article a new statistics for testing the equality of survival distributions with censored data are proposed. The tests are consistent against all  alternatives and with finite samples in situations of great clinical interest, such as the new oncological treatments where the new pharmacological strategies consist of introducing a delay effect \cite{melero2014therapeutic,xu2017designing} in the new drugs, greatly exceeding the performance of the classic tests if we select the correct parameters. In the other situations analyzed, the performance is higher as in the case of the study of healing, very close to the optimum when the hazard ratio is constant and slightly worse in the case of simulated multimodal treatments. In general, the performance is better than the classic tests, however there are certain issues such as the choice of optimal parameters or kernels in each situation that are still unresolved (It also happens in the uncensored case \cite{szekely2017energy}). In addition, in the analysis of survival when estimating the mean \cite{datta2005estimating}, it is common to consider the Efron correction \cite{efron1967two} that consists of considering that the maximum time observed in each group is uncensored ($\delta_{0(n_0:n_0)} = \delta_{0(n_1:n_1)}= 1$), or resorting to other imputation techniques with censored observations, both for the estimation of the mean \cite{datta2005estimating}, or in the global estimation of the weights of the Kaplan Meier estimator such as presmoothed \cite{cao2004presmoothed}. In any case, this may increase the power of the tests, but also increase the bias.
	
	
	The extension of the tests proposed with $k$-samples is analogous to the case without censorship, in which there is a variety of literature such as Disco analysis \cite{rizzo2010disco}, extension of the ANOVA test to testing the equality distribution in an uncensored context, or more recent the kernel methods proposed method in \cite{balogoun2018kernel} .
	
	Soon on my github at \url{https://github.com/mmatabuena} will appear  a R package  called energysurv with the proposed methods implemented in C++ in which the scientific community could use the new tests as a valuable alternative to classical survival tests.

	\section*{Graphics}

	\begin{figure}[h!]
		\centering
		\includegraphics[width=0.7\linewidth]{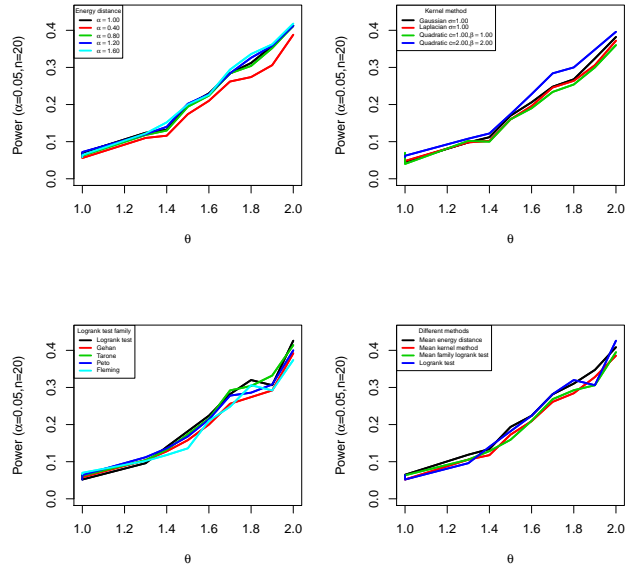}
		\caption{Power case of study: proportional hazard ratio in  two population $n=20$.}
		\label{fig:n1}
	\end{figure}

	\begin{figure}[h!]
		\centering
		\includegraphics[width=0.7\linewidth]{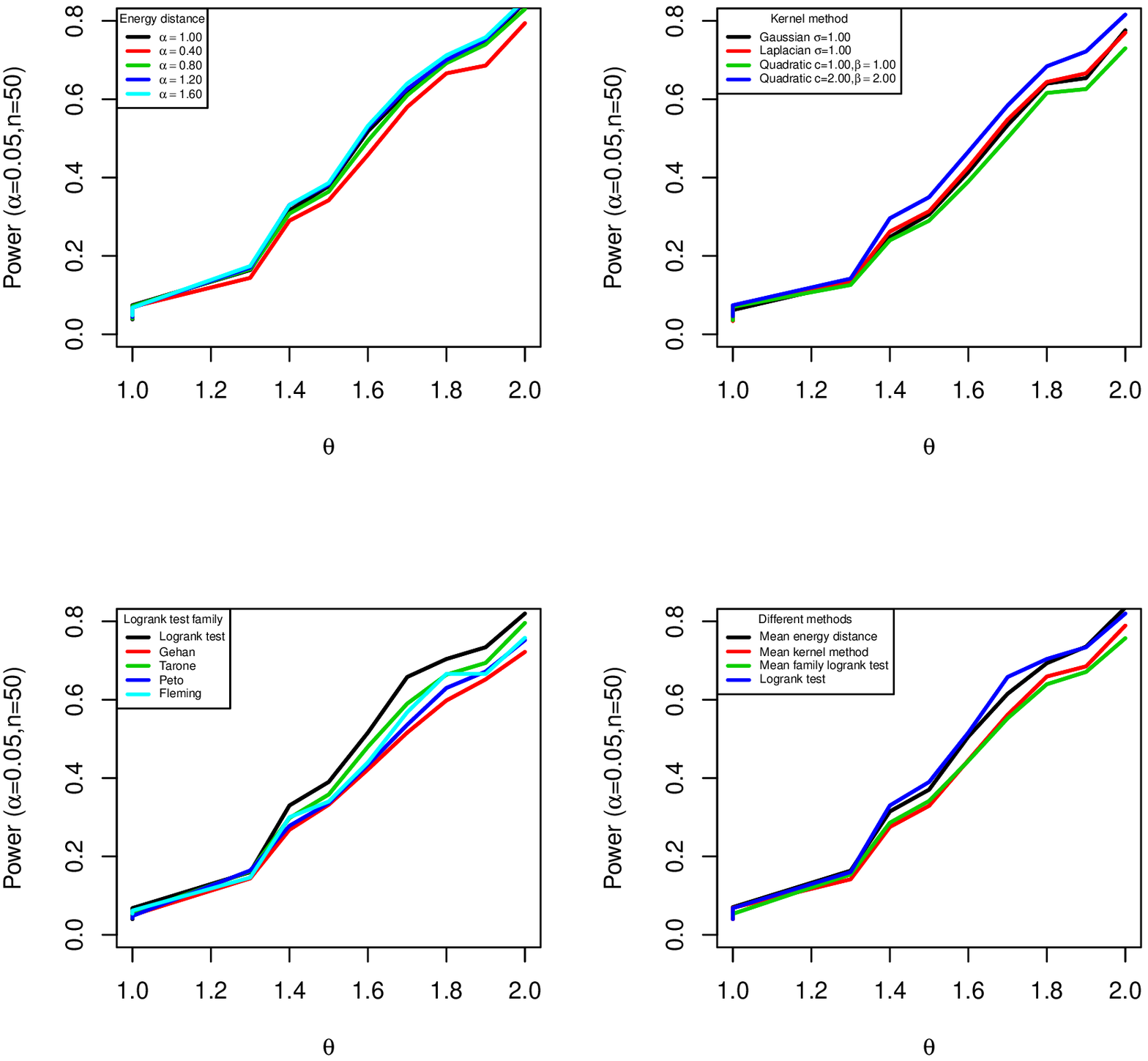}
		\caption{Power case of study: proportional hazard ratio in  two population $n=50$.}
		\label{fig:n2}
	\end{figure}

	\begin{figure}
		\centering
		\includegraphics[width=0.7\linewidth]{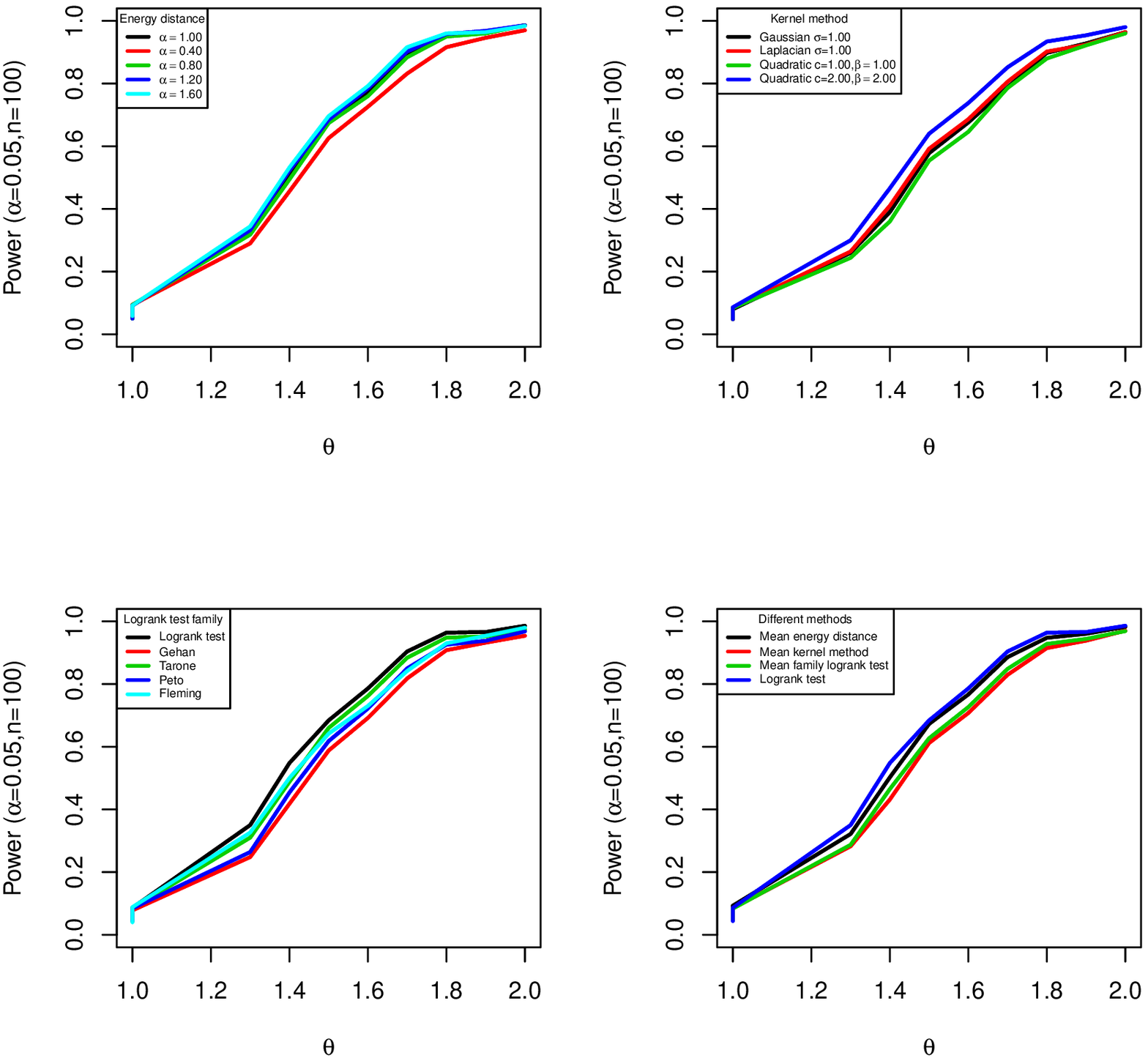}
		\caption{Power case of study: proportional hazard ratio in  two population $n=100$.}
		\label{fig:n3}
	\end{figure}

	\begin{figure}[h!]
		\centering
		\includegraphics[width=0.7\linewidth]{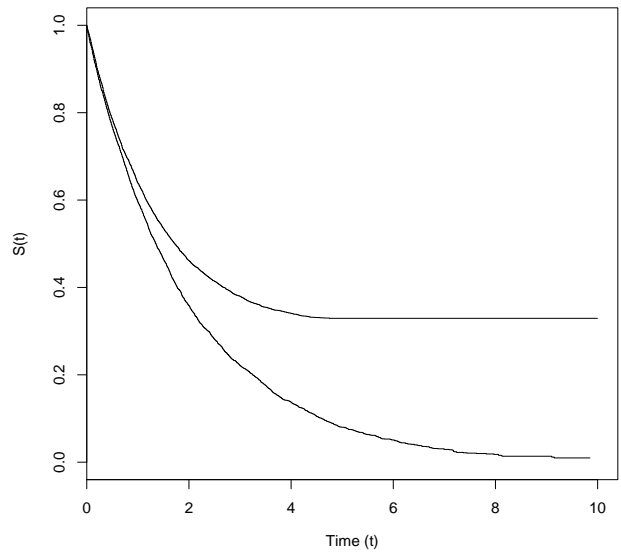}
		\caption{Survival curves case of study: healing ocurrs in a one  population.}
		\label{fig:n4}
	\end{figure}
	
	\begin{figure}
		\centering
		\includegraphics[width=0.7\linewidth]{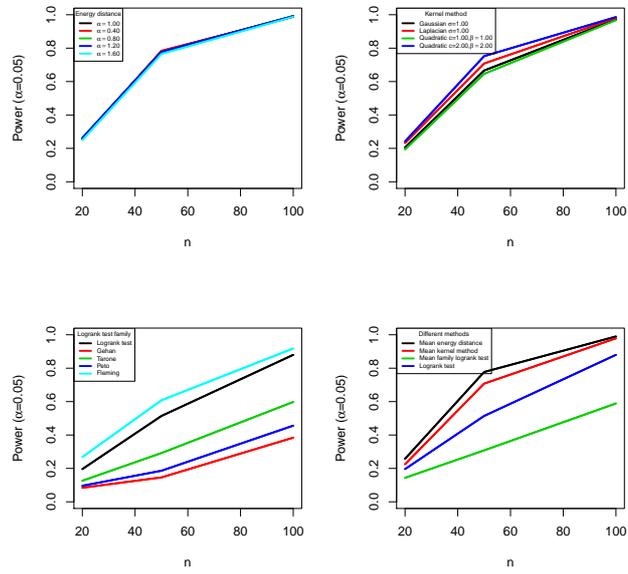}
		\caption{Power case of study: healing ocurrs in a one  population.}
		\label{fig:n5}
	\end{figure}

	\begin{figure}[h!]
		\centering
		\includegraphics[width=0.7\linewidth]{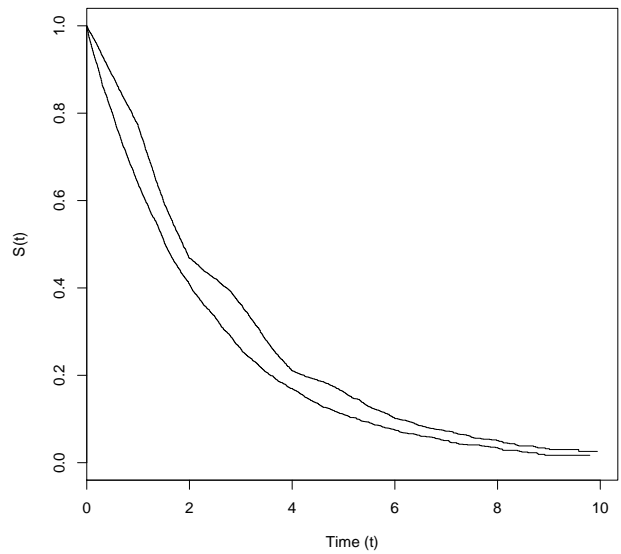}
		\caption{Survival curves case of study: multimodality treatment.}
		\label{fig:n6}
	\end{figure}

	\begin{figure}
		\centering

		\includegraphics[width=0.7\linewidth]{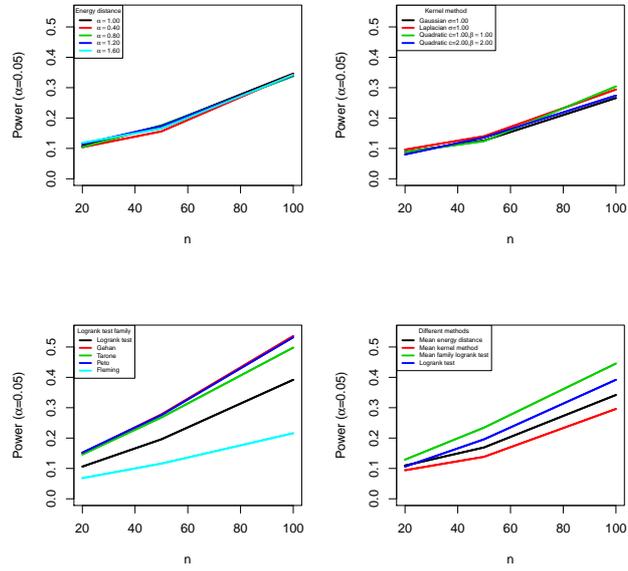}
		\caption{Power case of study: multimodality treatment.}
		\label{fig:n7}
	\end{figure}
	
	\begin{figure}[h!]
		\centering
		\includegraphics[width=0.7\linewidth]{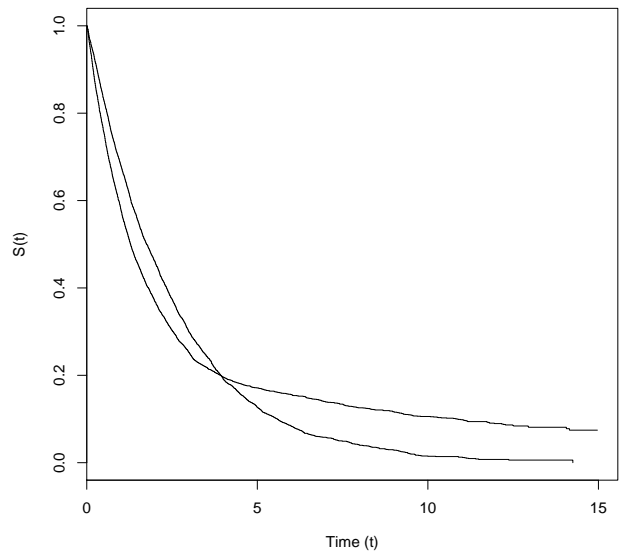}
		\caption{Survival curves case of study: delayed effects ocurrs in a one  population.}
		\label{fig:n8}
	\end{figure}
	
	\begin{figure}
		\centering
		
		\includegraphics[width=0.7\linewidth]{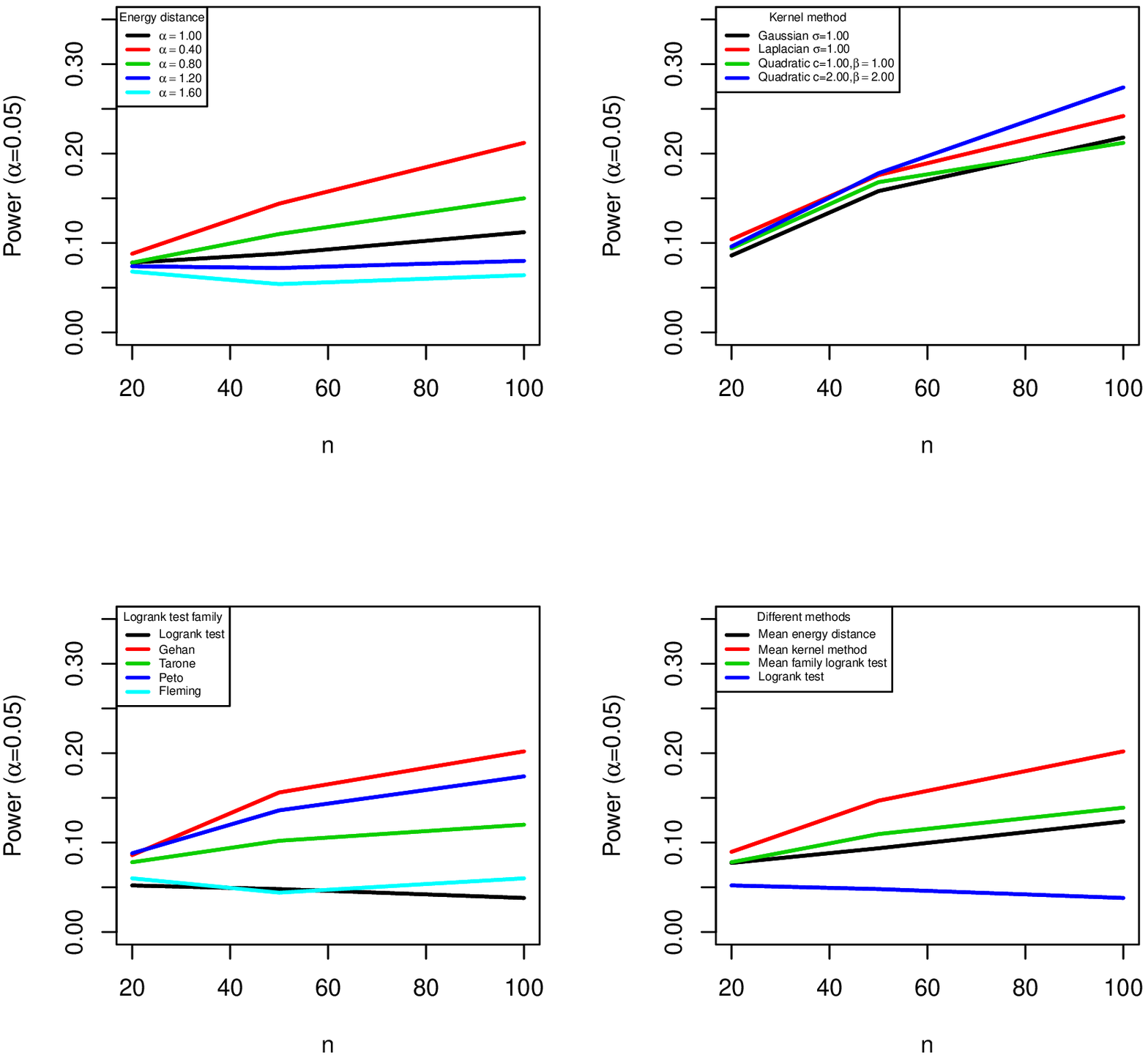}
		\caption{Power case of study: delayed effects in a one population.}
		\label{fig:n9}
		
	\end{figure}

	\section{Acknowledgements*}
	This work has received financial support from the Conseller\'ia de Cultura, Educación e Ordenaci\'on Universitar\'ia (accreditation 2016-2019, ED431G/08) and the European Regional Development Fund (ERDF).

\bibliographystyle{apalike}
\bibliography{bibliografia}

\end{document}